\documentclass[10pt]{amsart}
\usepackage{indentfirst,latexsym,bm}
\usepackage{amsfonts}
\usepackage{amssymb}
\usepackage{amsmath}
\usepackage{hyperref}
\usepackage{dsfont}
\usepackage{leftidx}
\usepackage{amsthm}
\usepackage{geometry}
\usepackage{amsmath,amscd}
\usepackage{tikz}
\usepackage[all,cmtip]{xy}
\usepackage{latexsym,amsmath,amssymb,amsfonts}  
\usepackage{color}
\usepackage{color,xcolor}
\usepackage{colordvi}
\usepackage{xspace}
\usepackage{graphicx}
\usepackage{pstricks}
\usepackage{pgf}
\usepackage{tikz}
\usepackage{xkeyval}
\usepackage{tikz}
\usetikzlibrary{mindmap,trees,shadows}

\begin{document}
 \thispagestyle{empty}

\title[Noetherian pointed Hopf algebras are affine]
{Noetherian pointed Hopf algebras are affine}

\author{
{Huan Jia $^{1}$}
and 
{Yinhuo Zhang $^{2}$}}

\address{$^{1}$ Department of Mathematics, Suqian Unviersity, Suqian City 223800, Jiangsu Province, China}

\address{$^{2}$ Department of Mathematics and Statistics,
University of Hasselt, Universitaire Campus, 3590 Diepenbeek, Belgium}

\address{Huan Jia $^{1}$}
\email{huan.jia@squ.edu.cn}

\address{Yinhuo Zhang $^{2}$}
\email{yinhuo.zhang@uhasselt.be}
 
\date{\today}
\maketitle

\newtheorem{theorem}{Theorem}[section]
\newtheorem{proposition}[theorem]{Proposition}
\newtheorem{lemma}[theorem]{Lemma}
\newtheorem{corollary}[theorem]{Corollary}
\theoremstyle{definition}
\newtheorem{definition}[theorem]{Definition}
\newtheorem{example}[theorem]{Example}
\newtheorem{remark}[theorem]{Remark}

\newcommand{\K}{\mathds{k}}
\newcommand{\Z}{\mathbb{Z}}

\newcommand{\N}{\mathds{N}}
\newcommand{\Pp}{\mathcal{P}}
\newcommand{\X}{\langle X \rangle}
\newcommand{\Y}{\langle Y \rangle}

\newcommand{\id}{\operatorname{id}}
\newcommand{\Char}{\operatorname{char}}
\newcommand{\gr}{\operatorname{gr}}
\newcommand{\ord}{\operatorname{ord}}
\newcommand{\corad}{\operatorname{corad}}
\newcommand{\lex}{\operatorname{lex}}
\newcommand{\LW}{\operatorname{LW}}
\newcommand{\G}{\operatorname{G}}

\newcommand{\R}{\operatorname{r}}
\newcommand{\rf}{\operatorname{rf}}
\newcommand{\irr}{\operatorname{irr}}
\newcommand{\ol}{o}

\theoremstyle{plain}
\newcounter{maint}
\renewcommand{\themaint}{\Alph{maint}}
\newtheorem{mainthm}[maint]{Theorem}

\theoremstyle{plain}
\newtheorem*{proofthma}{Proof of Theorem A}
\newtheorem*{proofthmb}{Proof of Theorem B}

\begin{abstract}
Let $\K$ be a field. In this paper, we introduce the notions of \emph{reduction order} and \emph{reduction-factorization} on words, and use them to show that any right or left Noetherian pointed Hopf algebra over $\K$ is affine. This result offers a partial affirmative answer to the classical affineness question for Noetherian Hopf algebras posed by Wu and Zhang \cite{WZ2003}. 
For a pointed Hopf algebra $H$ over $\K$, we construct a well-ordered set $X$ such that:
(1) $H$ is generated, as an algebra, by the subset $X_I$ of irreducible letters (with respect to the reduction order); and
(2) $X_I$ is finite whenever $H$ is right Noetherian.
\end{abstract}

\section*{Introduction}
Over the past three decades, substantial progress has been made toward understanding of infinite-dimensional Hopf algebras that satisfy various natural finiteness conditions---such as being affine Noetherian, having finite Gelfand-Kirillov dimension, being Artin-Schelter regular, or possessing the Calabi-Yau property; see, for instance, \cite{B1998, B2007, BG2014, G2013, BZ2020, A2023}.

In 2003, Wu and Zhang established that every affine Noetherian PI Hopf algebra is Artin-Schelter Gorenstein, and they proposed the following classical affineness question concerning Noetherian Hopf algebras:

\noindent
\textbf{Question.  \cite{WZ2003}} 
\textit{Is every Noetherian Hopf algebra over $\K$ an affine $\K$-algebra? }

To date, progress on this question remains limited. In 1975, Molnar \cite{M1975} proved the conjecture in the commutative and cocommutative cases:
\begin{itemize}
\item[(a)] A commutative Hopf algebra is Noetherian if and only if it is affine.
\item[(b)] A cocommutative Noetherian Hopf algebra is affine.
\end{itemize}
Beyond these foundational cases, the general question is still open---even for PI Hopf algebras, which are close to the commutative setting (see \cite{BG2014, G2013}). Significant partial results exist: over an algebraically closed field of characteristic zero, Goodearl and Zhang proved that a locally affine Noetherian Hopf algebra is affine if it is faithfully flat over any Hopf subalgebra \cite[Thm. 1.5(2)]{GZ2017}. This result implies that a Noetherian pointed Hopf algebra over an algebraically closed field of characteristic zero is affine.

In the graded setting, recent progress by Jia and Zhang \cite{JZ2025} shows that a Noetherian Hopf algebra graded as an algebra is affine whenever its degree-zero part is affine.  However, extending such results to non-graded Hopf algebras is significantly more challenging. A key obstacle is that the converse of the following classical lemma fails in general, even for Hopf algebras:

\noindent
\textbf{Lemma \cite{MR2001}} 
\textit{Let $R$ be a filtered ring. If the associated graded ring $\gr R$ of $R$ is Noetherian, then $R$ is Noetherian.}

This limitation motivates the development of new tools for addressing the affineness question in non-graded settings.
Pointed and connected Hopf algebras play a central role in this direction, as they are essential in the structure and classification theory of infinite-dimensional Hopf algebras; see \cite{Z2013, WZZ2015, WZZ2016, A2021, BGZ2019, ZSL2020, Z2024}.
Building on this framework, we introduce the notions of reduction order, reduction-factorization, and prime words.  Using these new tools, we are able to show that every right or left Noetherian pointed Hopf algebra is affine over an arbitrary field.   This result generalizes the previous work by Goodearl and Zhang, which established the same conclusion over an algebraically closed field of characteristic zero \cite[Thm. 1.5(3)]{GZ2017}.

The paper is organized as follows. Let $(X,\prec)$ be a well-ordered set, $\X$ the free monoid on $X$, and $\K\X$ the free algebra on $X$. In Section \ref{sec:1}, we review basic background on Hopf algebras and the lexicographic order.

In Section \ref{sec:2}, we introduce the reduction order on $\X$ (Definition \ref{def:reduction-order}), along with reduction-factorization and prime words (Definition \ref{def:reduction-factorization}).
We first establish that the reduction order defines a well-ordering on $\X$ (Lemma~\ref{lem:reduction-order-well}) and present several examples to illustrate its behavior. We then show that every prime word occupies a distinguished position within this order: specifically, each prime word is the greatest among all of its own factors with respect to the reduction order (Lemma \ref{lem:properties-alpha_R}).

In Section~\ref{sec:3}, we investigate irreducible words with respect to the reduction order. We show that every prefix of an irreducible word is also irreducible, whereas its suffixes need not be (Proposition~\ref{prop:prefix-irrws}, Remark~\ref{rmk:suffix-irrws-false}). We further prove that if an augmented algebra is right Noetherian, then it contains only finitely many irreducible letters (Lemma~\ref{lem:Noetherian-irrls-finite}).

In Section \ref{sec:4}, using the reduction order, reduction-factorization and prime words, we describe the comultiplication of arbitrary words in $\X$ (Lemmas \ref{lem:coproduct-omega-1}--\ref{lem:coproduct-omega-2}).

In Section \ref{sec:5}, building on this comultiplication formula and Lemma \ref{lem:wr-sum-smaller-words}, we prove that every reducible letter in $X$ can be written as a linear combination of products of letters of strictly smaller order (Lemma \ref{lem:rl-sum-smaller-irrls}).

As a consequence, every non-empty word in $\X$ can be expressed as a linear combination of products of irreducible letters---an essential step toward establishing the affineness of Noetherian pointed Hopf algebras.

\noindent
\textbf{Theorem A.} (Theorem \ref{thm:rl-sum-smaller-irrls}) 
\textit{Suppose that $(X = O \cup D^{*}, \prec)$ is a well-ordered set together with a map $t: X \rightarrow \mathbb{N}$ satisfying $t(x^{*}) = t(x)$ for $x \in D$, where $D\subseteq O$ and $D^*$ is the mirror set of $D$.
Let $(\K\X,\epsilon)$ be the augmented free algebra with $\epsilon(X)=0$, and let $I \subseteq \ker \epsilon$ be a proper ideal of $\K\X$. Suppose $\Delta$ is a skew-triangular comultiplication on $\K\X$ such that $\Delta(I) \subseteq I \otimes \K\X + \K\X \otimes I$.
Assume further that for every $z \in X_{0}$, there exists $c_{z} \in \K\X$ satisfying:
\begin{itemize}
\item $c_{z}(1-z), (1-z)c_{z} \in 1 + I\ (c_z$ is the inverse of $1-z$ modulo $I)$,
\item $\ol(m_{c_{z}}) \preceq \ol(z)$ $($order control condition$)$.
\end{itemize}
If $x \in X$ is an $I$-reducible letter, then there exist $I$-irreducible letters $x_{1}, \ldots, x_{n} \in X$ }

\textit{and scalars $k_{x_{1}\cdots x_{n}} \in \K$ such that
\begin{align*}
x \in \sum_{} k_{x_{1}\cdots x_{n}}x_{1} \cdots x_{n} + I, \qquad x_{1},\ldots,x_{n} \prec x.
\end{align*}
Consequently, every non-empty word in $\X$ can be expressed modulo $I$ as a linear combination of products of $I$-irreducible letters.}

In Section~\ref{sec:6}, we apply the results from Section~\ref{sec:5} to the setting of pointed Hopf algebras. We prove that every right or left Noetherian pointed Hopf algebra is affine. Let $H$ be a pointed Hopf algebra over $\K$. We construct a generating set $X$ of $H$. Denote by $\X$ the free monoid on $X$ and by $\K\X$ the corresponding free algebra.  Let $\pi: \K\X \rightarrow H$ be the canonical projection, and write $I=\ker\pi$.  Denote by $X_{I}$ the set of $I$-irreducible letters (irreducible with respect to the reduction order introduced in \S 2--\S 3).
Theorem A  together with Lemmas \ref{lem:Noetherian-irrls-finite} and \ref{lem:relation-x-x*}  yields  the following  result. 

\noindent
\textbf{Theorem B.} (Theorem \ref{thm:noetherian-pointed-affine}) 
\textit{Let $H$ be a pointed Hopf algebra over $\K$. The following hold:
\begin{itemize}
    \item [(a)] 
There exists a generating set $X$ of $H$ such that $\pi(X_{I})$ generates $H$ as an algebra, where $\pi: \K\X \rightarrow H$ is the canonical projection and  $I=\ker\pi$.
\item [(b)] If $H$ is right or left Noetherian, then $H$ is affine.
\end{itemize}
}

In particular, every right or left Noetherian connected Hopf algebra is affine.
Another interesting consequence of Theorem B is that every Hopf subalgebra of a right or left Noetherian pointed Hopf algebra is affine (Corollary \ref{cor:Hopfsubalg-affine}), extending classical results for group algebras and universal enveloping algebras (Remark \ref{rmk:Hopfsubalg-affine}). 

Several ideas in this paper draw inspiration from \cite{ZSL2020,J2023,K1999,R1999,U2004}, including the use of irreducible words under graded lexicographic order, the role of comultiplication in PBW-type arguments, and the use of word length in ordering.

We expect that the methods developed here, such as the reduction order, reduction-factorization, prime words, and the treatment of invertible elements---will be useful in further investigations of the affineness problem and related questions. In fact, these techniques have recently been applied in \cite{JZ2025-2} to show that Noetherian connected braided Hopf algebras are affine.

\section{Coalgebras, Hopf algebras and the lexicographic order}\label{sec:1}
Throughout, let $\K$ be a field, $\K^{*} = \K \setminus \{0\}$, $\Z$ the ring of integers, and $\mathbb{N}$ the set of natural numbers. For a group $G$, denote by $\ord(a)$ the order of an element $a \in G$.  A ring (respectively, a $\K$-algebra) is called \textit{affine} if it is finitely generated as a ring (respectively, as a $\K$-algebra).

\subsection{Coalgebras and Hopf algebras}
\qquad

Following \cite{S1969,Ra2012}, we recall some basic concepts concerning  coalgebras and Hopf algebras. 

Let $C$ be  a coalgebra over $\K$, denote by $\G(C)$  the set of all \textit{group-like elements} of $C$, that is, $\G(C):=\{g\in C\mid \Delta_{C}(g)=g\otimes g, ~\epsilon_{C}(g)=1\}$. For $g,h\in \G(C)$, define  the set of $(g,h)$-\textit{skew primitive elements}  by
$$\Pp_{g,h}(C):=\{x\in C\mid \Delta_{C}(x) = g\otimes x + x\otimes h,~\epsilon_{C}(x)=0\}.$$
In particular, the elements of $\Pp_{1,1}(C)$ are called \textit{primitive}, we set  $\Pp(C):=\bigcup_{g,h\in \G(C)}\Pp_{g,h}(C)$.  The \textit{coradical} of $C$, denoted $\corad(C)$, is the sum of all simple subcoalgebras of $C$. The coalgebra $C$ is said to be  \textit{pointed} (resp. \textit{connected}) if $\corad(C)=\K\G(C)$ (resp. $\corad(C)=\K$). For a Hopf algebra $H$ over $\K$, denote by $S_{H}$ the antipode of $H$.

Recall that the \textit{coradical filtration} $\{C_{(n)}\}_{n=0}^{\infty}$ of a coalgebra $C$, defined recursively by 
$$C_{(0)}=\operatorname{corad}{C}  \text{ and }  C_{(n)}:=C_{(n-1)} \bigwedge C_{(0)}= \Delta_{C}^{-1}(C_{(n-1)} \otimes C + C \otimes C_{(0)}),\ n\ge 1,$$
is a coalgebra filtration of $C$.

For a Hopf algebra $H$, the coradical filtration $\{H_{(n)}\}_{n\ge 0}$ of $H$  is a Hopf algebra filtration of $H$ if and only if the coradical $H_{(0)}$  is a Hopf subalgebra of $H$, see \cite{Ra2012} for the details. Consequently, the coradical filtration of a pointed Hopf algebra is a Hopf algebra filtration.

We next  recall a result concerning the coradical of a pointed coalgebra.
\begin{lemma}\cite{Ra2012}\label{lem:corad-pointed-Hopfalg}
Let $C$ be a pointed coalgebra over $\K$ and fix an element $a\in \G(C)$. Then 
\begin{align*}
C_{(0)} = \corad(C)= \K\G(C) = \K a \bigoplus (\bigoplus_{a\neq g\in \G(C)} \K(a-g)).
\end{align*}
\end{lemma}

Recall from \cite[Corollary 3.5.15]{Ra2012} that for a pointed coalgebra $C$ over $\K$, there exists an orthonormal family $\{e_{g}\}_{g\in \G(C)}$ of idempotents  of $C^{*}$ satisfying $e_{g}(h)= \delta_{g,h}$ for all $g,h\in \G(C)$. 
For any subspace $U\subseteq C$ and $g,h\in \G(C)$, denote by $$U^{h}:= e_{h} \rightharpoonup  U, \  \ {^{g}U}:= U \leftharpoonup e_{g}, \text{ and } \ {^{g}U^{h}}:= e_{h} \rightharpoonup  U \leftharpoonup e_{g},$$ where the left and right actions of $C^*$ on $C$ are defined as follows: for $u\in U$ and $c^{*} \in C^{*}$,
$$c^{*} \rightharpoonup u = \sum_{} u_{(1)} c^{*}(u_{(2)}), \text{ and }  u \leftharpoonup c^{*} = \sum_{} c^{*}(u_{(1)}) u_{(2)}.$$

Then the following lemma describes the coproduct of a pointed coalgebra.

\begin{lemma}\cite[Proposition 4.3.1]{Ra2012}\label{lem:coproduct-pointed-coalgebra}
Let $C$ be a pointed coalgebra over $\K$. The following statements hold:
\begin{itemize}
    \item [(a)] If $n\ge 1$ and $c\in {^{g}(C_{(n)})^{h}}$, then 
    \begin{align*}
    \Delta_{C}(c) =  g \otimes c + c \otimes h + v, \qquad v\in  \sum_{a\in \G(C)} {^{g}(C_{(n-1)})^{a}} \otimes {^{a}(C_{(n-1)})^{h}};
    \end{align*}
    \item [(b)] If $g\neq h$, then $\mathcal{P}_{g,h}= \K(g-h) \oplus {^{g}(C_{(1)})^{h}}$;
    \item [(c)] ${^{g}(C_{(1)})^{g}} = \mathcal{P}_{g,g}(C) \oplus \K g$;
\end{itemize}
\end{lemma}

\subsection{The lexicographic order on words}\quad

Let $X$ be a set. Denote by $\X$  the free monoid generated by $X$, and by $\K\langle X\rangle$ the free associative algebra generated by $X$ over $\K$. The empty word in $\X$ is denoted by $1$.  Let $\X^{+}$ be the free semigroup generated by $X$. Clearly, $\X = \X^{+} \cup \{1\}$, and  $\K\X^{+}$ is the subspace of $\K\X$ spanned by $\X^{+}$.

A word $\beta\in \X$ is called a \textit{factor} of $\alpha\in \X$ if there exist words $\gamma,\eta\in \X$ such that $\alpha=\gamma\beta\eta$. If $\gamma=1$ (resp. $\eta=1$), then $\beta$ is called a \textit{prefix} (resp. \textit{suffix}) of $\alpha$.
If $\beta\neq \alpha$, then  $\beta$ is called \textit{proper} factor of $\alpha$. 

For a word $\alpha= x_{1} \cdots x_{n} \in \X$ with letters $x_{1},\ldots,x_{n}\in X$,  define
\begin{itemize}
\item  $\alpha(y)$ as the number of occurrences of a letter $y\in X$ in $\alpha$, that is, \\
 $\alpha(y):=\# \{i \mid x_{i}=y,~1\le i\le n\}$;
\item  $l(\alpha)$ as the first letter of  $\alpha$, i.e.  $l(\alpha)=x_{1}$; 
\item  $|\alpha|$ as the \textit{length} of  $\alpha$, i.e. $|\alpha|=n$. 
\end{itemize}
 
For each integer $m\ge 1$, write $X^{m}:= \{\alpha \in \X \mid |\alpha| = m\}$.

Given a well-ordered set $(X,\prec)$, the \textit{lexicographic order} $\prec_{\lex}$ on the free monoid $\X$ is defined as follows (see \cite{L1997}). For any words $\alpha,\beta\in \X$, 
\begin{align*}
\alpha\prec_{\lex} \beta \Longleftrightarrow 
\left\{\begin{array}{l}
\alpha \text { is a proper prefix of } \beta, \text { or } \\
\alpha=\gamma x \eta,~\beta=\gamma y \zeta \text { with } x, y \in X, ~x\prec y \text{ and } \gamma, \eta, \zeta \in\langle X\rangle.
\end{array}\right.
\end{align*}

\begin{example}
Let $X=\{x,y\}$ with $x \prec y$. Then, by definition,  we have
$$x \prec_{\lex} x^{2} \prec_{\lex} x^{3} \prec_{\lex}  xy \prec_{\lex} y \prec_{\lex} yx \prec_{\lex} y^{2}.$$
\end{example}

\begin{remark}
The lexicographic order is not a well order on $\X$, even when $X$ is a finite set. 
For example, let $X=\{x,y\}$ with the order $x\prec y$, and define
$$B:=\{x^{n}y \mid n\ge 0\}.$$
Then the subset $B$ of $\X$ has no least element, because $x^{n+1}y\prec_{\lex} x^{n}y$ for $n\ge 0$. Hence, $(\X, \prec_{\lex})$ is not well-ordered.
\end{remark}

\section{The reduction order and the reduction-factorization on words}\label{sec:2}

To establish the main theorems of this paper, we introduce in this section the concepts of \textit{the reduction order}, the associated \textit{reduction-factorization} on words, and prime words.

Let $(O,<_{O})$ be a well-ordered set and let $D$ be a subset of $O$. Define
$$
D^{*}:=  \{x^{*} \mid x \in D\},\ \  X := O \cup D^{*}.
$$
We refer to $O$ as the \textit{original subset} of $X$ and to $D^{*}$ as  the \textit{mirror set} of $D$. In the special case where $D=\emptyset$, we have $X = O$. 

Define a map $\ol: \X^{+} \rightarrow O$ by setting, for any $\omega\in \X$,
\begin{align*}
\ol(x\omega) &= x, \text{ if } x \in O;\\
\ol(x^{*}\omega) &= x, \text{ if } x^{*} \in D^{*} \text{ with } x\in D.
\end{align*} 
Clearly, $\ol(x) = x$ for $x \in O$, and  $\ol(x^{*}) = x$ for $x^{*} \in D^{*} \text{ with } x\in D$. For $x\in X$, we call $\ol(x)$ the \textit{original element} of $x$.
Hence, for a non-empty word $\alpha\in \X$, the element  $o(\alpha)$ is defined to be  the original element of the first letter $l(\alpha)$. When $l(\alpha) \in O$, we have $\ol(\alpha) = l(\alpha)$; when $l(\alpha) \in D^{*}$, $\ol(\alpha) \neq l(\alpha)$.

Now we define an order $\prec$ on  $X$ as follow. For $x,y \in X$,
\begin{align*}
x \prec y 
\Longleftrightarrow 
\left\{\begin{array}{l}
 \ol(x) <_{O}   \ol(y), \text{ or } 
 \\
\ol(x) = \ol(y) = x \text{ and } y=x^{*}.
\end{array}\right.
\end{align*}
It is straightforward to verify that $(X, \prec)$ is a well-ordered set. Moreover, the order $\prec$ is compatible with the order $<_{O}$ on $O$; that is, $x \prec y$ if and only if $x <_{O} y$ for all $x, y \in O$.
By definition, we clearly have $\ol(\alpha) \preceq l(\alpha)$ for any $\alpha \in \X$.

Next,  we  introduce an order on the set of words $\X$, called the reduction order on $\X$:
\begin{definition}\label{def:reduction-order}
The \textit{reduction order} $\prec_{\R}$ on $\X$ is defined as follows. For any non-empty words $\alpha,\beta\in \X$, $1 \prec_{\R} \alpha$, and 
\begin{align*}
\alpha \prec_{\R} \beta \Longleftrightarrow 
& \left\{\begin{array}{l}
\ol(\alpha)\prec \ol(\beta), \text { or } \\
\ol(\alpha )= \ol(\beta) \text{ and } 
\left\{\begin{array}{l}
|\alpha| < |\beta|, \text{ or } \\
|\alpha| = |\beta| \text{ and } \alpha \prec_{\lex} \beta.
\end{array}\right.
\end{array}\right.
\end{align*}
\end{definition}

It is easy to verify that the reduction order is a total order on $\X$. Moreover, for any non-empty word $\alpha\in \X$, every proper prefix $\beta$ of $\alpha$ satisfies $\beta \prec_{\R} \alpha$. 
Hence, each word $\alpha$ is the greatest among all its prefixes.

\begin{example}\label{eg1:reduction-order}
Let $X=\{x,y\}$ with $D=\emptyset$ and $x \prec y$. Observe that $o(xy) = x = o(x^{3})$ and $|xy| = 2 < 3 = |x^{3}|$. By definition, we have $xy \prec_{\R} x^{3}$, even though $x^{3} \prec_{\lex} xy$. 
\end{example}

\begin{example}\label{eg:x-x*}
Let $X=\{x,x^{*}\}$ with $D=\{x\}$, the mirror set $D^{*} = \{x^{*}\}$ and the order  $x \prec x^{*}$. Note that 
\begin{align*}
& \ol(x^{*}x) = \ol(xx^{*}) = \ol(x^{*}) = \ol(x) = x, \quad |x| = |x^{*}|  < |xx^{*}| = |x^{*}x|, \quad xx^{*} \prec_{\lex} x^{*}x.
\end{align*}
Hence, by definition  
$$x \prec_{\R} x^{*} \prec_{\R} xx^{*} \prec_{\R} x^{*}x.$$
However, with respect to the lexicographic order, we have
$$x \prec_{\lex} xx^{*} \prec_{\lex} x^{*} \prec_{\lex} x^{*}x.$$
\end{example}

Now we show that the reduction order $\prec_{\R}$ is a well order on the set $\X$.
\begin{lemma}\label{lem:reduction-order-well}
The set $(\X, \prec_{\R})$ is well-ordered;  that is, every non-empty subset of $\X$ has a least element with respect to $\prec_{\R}$.
\end{lemma}

\begin{proof}
It suffices to show that every non-empty subset $G\subseteq \X$ has a least element with respect to $\prec_{\R}$. Let
\begin{align*}
L:=\{o(\omega) \mid \omega\in G\}.
\end{align*}
Since $L\subseteq X$ and $X$ is well-ordered, $L$ has a least element, say $z$. Define
\begin{align*}
L_{1}:=\{\omega\in G \mid o(\omega)=z\}.
\end{align*}
Among the words in $L_1$, choose those with minimal length $m$, and set 
\begin{align*}
L_{2}:=\{\omega\in G \mid o(\omega)=z, ~|\omega|=m\}.
\end{align*} 
Since $(X^{m},\prec_{\lex})$ is well-ordered and $L_{2} \subseteq X^{m}$, 
the set  $L_{2}$ has a least element with respect to the lexicographic order, say
$$\alpha = z\beta\ \text{or}\ z^{*}\beta,$$
where $\beta$ is some word (the second case occurs if $z\beta \notin G$). By the definition of the reduction order  $\prec_{\R}$, it follows that $\alpha$ is the least element in $G$. Hence, $(\X, \prec_{\R})$ is well-ordered.
\end{proof}

It is straightforward to verify that the two orders $\prec_{\R}$ and $\prec$ are compatible.
\begin{lemma}\label{lem:prec_R-compatible-prec}
The reduction order $\prec_{\R}$ is compatible with the order $\prec$  on the set $X$, that is, for any letters $x,y\in X$,  we have $x \prec_{\R} y$ if and only if $x \prec y$.
\end{lemma}

Applying Lemma \ref{lem:prec_R-compatible-prec} to words, we obtain the following result:
\begin{corollary}\label{cor:prec_R-compatible-prec-2}
Let $\alpha$ be a non-empty word in $\X$,  and let $y$ be a letter in $O$. The following statements are equivalent: 
\begin{enumerate}
\item[(a)]$ \alpha \prec_{\R} y$;
\item[(b)]  $l(\alpha) \prec y$;
\item[(c)] $\ol(\alpha) \prec y$.
\end{enumerate}
\end{corollary}

\begin{proof}
(a) $\Rightarrow$ (b).
Since $\alpha$ is the greatest among its prefixes, we have $l(\alpha) \preceq_{\R} \alpha \prec_{\R} y$.
By Lemma \ref{lem:prec_R-compatible-prec}, this implies $l(\alpha) \prec y$.

(b) $\Rightarrow$ (c).
Because $\ol(\alpha) \preceq l(\alpha)$ and $l(\alpha) \prec y$, it follows that $\ol(\alpha) \prec y$.

(c) $\Rightarrow$ (a).
As $y \in O$, we have $\ol(y) = y$.
Hence $\ol(\alpha) \prec \ol(y)$, which, by the definition of $\prec_{\R}$, yields $\alpha \prec_{\R} y$.
\end{proof}

We now state a useful property of the reduction order on words.

\begin{lemma}\label{lem:ac<bc}
Let $\alpha, \beta$ be non-empty words in $\X$ with $\alpha \prec_{\R} \beta$. Then for any word $\gamma \in \X$, one has
$\alpha\gamma \prec_{\R} \beta\gamma$.
\end{lemma}

\begin{proof}
The proof follows directly from the definition of the reduction order. We distinguish three cases.

\noindent
\textit{Case 1}: $\ol(\alpha) \prec \ol(\beta)$. Then we obtain $\ol(\alpha\gamma) \prec \ol(\beta\gamma)$ and $\alpha\gamma \prec_{\R} \beta\gamma$.\\
\textit{Case 2}: $\ol(\alpha) = \ol(\beta)$ and $|\alpha|<|\beta|$. Then $\ol(\alpha\gamma) = \ol(\beta\gamma)$ and $|\alpha\gamma|<|\beta\gamma|$. So $\alpha\gamma \prec_{\R} \beta\gamma$.\\
\textit{Case 3}: $\ol(\alpha) = \ol(\beta)$, $|\alpha|=|\beta|$ and $\alpha \prec_{\lex} \beta$. Then we have $\ol(\alpha\gamma) = \ol(\beta\gamma)$, $|\alpha\gamma|=|\beta\gamma|$,  and $\alpha\gamma \prec_{\lex} \beta\gamma$. Therefore,  $\alpha\gamma \prec_{\R} \beta\gamma$.
\medskip
In all cases, the  desired conclusion holds. 
\end{proof}

\begin{remark}
The reduction order is not left-compatible in general: it need not hold that   $\gamma \alpha \prec_{\R} \gamma\beta$ whenever $\alpha \prec_{\R} \beta$  for any word $\gamma\in \X$. 
For example, let $X=\{x,y,z\}$ with $x\prec y \prec z$.  We have $x^{2} \prec_{\R} y$,  but $zx^{2} \succ_{\R} zy$.
\end{remark}

By definition and Lemma \ref{lem:prec_R-compatible-prec}, we have the following result for letters of $X$.
\begin{lemma}\label{lem:x-least}
Let $x\in O$ and set $L_{x} := \{\alpha \in \X \mid  \ol(\alpha) = x\}$. Then 
\begin{itemize}
\item [(a)] For distinct $y, z \in O$, we have $L_{y} \cap L_{z} = \emptyset$ and $\X^{+} =  \bigcup_{x\in O} L_{x}$.
\item [(b)] The letter $x$ is the least element of the set $L_{x}$ w.r.t. $\prec_{\R}$. 
\item [(c)] If $x\in D$, then $x^{*} \in D^{*}$ is the second least element of $L_{x}$  w.r.t. $\prec_{\R}$.
\item [(d)] A letter $a \in X$ is the least non-empty word of $\X$ (w.r.t. $\prec_{\R}$) if and only if $a$ is the least letter of $X$ w.r.t. $\prec$. Equivalently,  $a$ is the least letter of $O$ w.r.t. $\prec$ (or $<_{O}$).
\end{itemize}
\end{lemma}

Next we define the factorization of a word w.r.t. the reduction order:

\begin{definition}\label{def:reduction-factorization}
Let $\alpha$ be a  non-empty word  in $\X$.   The \textit{reduction-factorization} (or \textit{r-factorization} for short) of $\alpha$ is the decomposition
$$\alpha=\alpha_{L}\alpha_{R}$$
 such that $\alpha_{R}$ is the greatest suffix of $\alpha$ w.r.t. the reduction order. We write this factorization of $\alpha$ as 
\begin{align*}
\rf(\alpha)=(\alpha_{L},\alpha_{R}).
\end{align*}
A non-empty word $\alpha$ is called a \textit{prime word} if $\alpha_{R} = \alpha$ or $\alpha_{L}= 1$.  We denote by $\X_{p}$ the set of all prime words in $\X$.
\end{definition}

It is clear that the reduction-factorization of a word is unique. Moreover, every letter $x\in X$ is a prime word because $x_{R}=x$ and $x_{L} = 1$.

\begin{example}\label{eg1:xR}
Let $X=\{x,y,z\}$ with  $x \prec y \prec z$. 
By  definition we have 
$$(x^{3})_{R}=x^{3}, ~(xy)_{R}=y,\  \text{and}\ (yx)_{R}=yx.$$
Hence, the r-factorization of a word is closely related to the greatest letter occurring in it.  Moreover, note that  $xxy \succ_{\R} xz$,  while 
$$(xxy)_{R}=y\prec_{\R} z =(xz)_{R}.$$ 
This shows that a word of higher reduction order may have a suffix of lower reduction order, and conversely, a smaller word can have a larger suffix. 
\end{example}

\begin{example}\label{eg2:xx*R}
Let $X=\{x,x^{*}\}$, $D=\{x\}$, and $D^{*} = \{x^{*}\}$ with $x \prec x^{*}$. 
Recall from Example \ref{eg:x-x*} that 
$$x \prec_{\R} x^{*} \prec_{\R} xx^{*} \prec_{\R} x^{*}x.$$
 Then
\begin{align*}
(xx^{*})_{R} = xx^{*}, \quad (x^{*}x)_{R} = x^{*}x.
\end{align*} 
This situation differs from Example \ref{eg1:xR} since here $x$ is not the greatest letter of the word $xx^{*}$. 
\end{example}

For a word $\alpha \in \X$, denote by $m_{\alpha}$ the greatest letter appearing in $\alpha$.  For a polynomial $p=\sum_{i}k_{i}\alpha_{i} \in \K\X$ with $\alpha_{i}\in \X$, denote by $m_{p}$ the greatest letter of all the $\alpha_{i}$'s. 

By Examples \ref{eg1:xR}--\ref{eg2:xx*R}, we see that the r-factorization of a word is closely tied to the original position and type of its greatest letter. In general, we have the following results describing the behavior of r-factorizations and prime words.

\begin{lemma}\label{lem:properties-alpha_R}
Let $\alpha$ be a non-empty word in $\X$, and let $\rf(\alpha)=(\alpha_{L},\alpha_{R})$ be its reduction-factorization. Set $m:=\ol(m_{\alpha})$. Then the following statements hold:
\begin{itemize}
    \item [(a)] $\alpha_{L}(m)=0$, and $\alpha_{L}(m^{*})=0$ if $m\in D$; moreover, $o(\alpha_{R}) = o(m_{\alpha_{R}}) = m$. 
    \item [(b)] $o(\alpha)= m$ if and only if $\alpha_{R}=\alpha$.
    \item [(c)] A factorization $\alpha=\alpha_{1}\alpha_{2}$ is the r-factorization of $\alpha$ if and only if $\alpha_{1}(m)=0$, $\alpha_{1}(m^{*})=0$ when $m\in D$, and $o(\alpha_{2})=m$.
    \item [(d)] The suffix $\alpha_{R}$ is the greatest factor of itself w.r.t. the reduction order. Consequently,  $\alpha_{R}$ is the greatest factor of $\alpha$ w.r.t. the reduction order.
\end{itemize}
\end{lemma}

\begin{proof}
Parts (a)-(c) follow directly from the definition of the reduction-factorization.

For Part (d),  assume to the contrary that $\alpha_{R}$ is not the greatest factor of itself. Then there exists a factor $\beta$  of $\alpha_{R}$ such that $\alpha_{R} = \gamma\beta\eta$ with $\gamma, \eta\in \X$, and $\beta$ strictly greater than $\alpha_{R}$ in the reduction order. 
Since $\beta\eta \succeq_{\R} \beta$, it follows that  $\eta=1$,   hence $\alpha_{R}= \gamma\beta$. 

Assume now $\gamma\neq 1$. 
Note that $|\beta|<|\gamma\beta|$,  and by Part (a) we have 
$$\ol (\beta) \preceq \ol (m_{\alpha_{R}}) = \ol(\alpha_{R}) = \ol(\gamma\beta).$$
Therefore,  we have $\beta \prec_{\R} \gamma\beta$,  contradicting the assumption that  $\beta$ is the greatest factor of $\alpha_{R}$. Thus, $\gamma=1$ and $\beta=\alpha_{R}$, proving that $\alpha_{R}$ is indeed the greatest factor of itself.

Finally, since $\alpha_{L}(m)=0$ and $\alpha_{L}(m^{*})=0$ (for $m \in D$),  it follows that no factor of $\alpha_{L}$ can exceed  $\alpha_{R}$ in the reduction order, so $\alpha_R$ is also the greatest factor of $\alpha$.
\end{proof}

By Lemma \ref{lem:properties-alpha_R},  the following property holds for prime words:

\begin{proposition}\label{prop:prime-words}
Let $w_{1}, w_{2}\in \X_{p}$ be prime words such that  $w_{1} \succeq_{\R} w_{2}$. Then the product $w = w_{1}w_{2}$ is also a prime word,  and $w \succ_{\R} w_{1} \succeq_{\R} w_{2}$. 
\end{proposition}

We show that every non-empty word possesses a unique factorization by prime words.
\begin{proposition}\label{prop:unique-factorization}
Every non-empty word $\omega$ in $\X$ admits a unique factorization as a strictly ascending product of prime words with respect to the reduction order, that is,
\begin{align*}
\omega = \omega_{n}\omega_{n-1}\cdots \omega_{1}, \qquad \omega_{n},\ldots,\omega_{1} \in \X_{p}, \qquad \omega_{n}\prec_{\R} \ldots \prec_{\R} \omega_{1}.
\end{align*} 
\end{proposition}

\begin{proof}
Set $\omega_{1} = \omega_{R}$,  and define  inductively 
\begin{align*}
\omega_{i} = ((\ldots (\omega\underbrace{_{L})_{L}\ldots)_{L}}_{i-1})_{R} \quad \text{ if  } (\ldots (\omega\underbrace{_{L})_{L}\ldots)_{L}}_{i-1} \neq 1, \quad i \ge 2.
\end{align*}
Clearly, each $\omega_{i}$ lies in $\X_{p}$.
By Lemma \ref{lem:properties-alpha_R} (d), we have $\omega_{i+1} \prec_{\R} \omega_{i}$ for all $i$,  which establishes the existence of the desired factorization.  To prove the uniqueness, suppose that 
\begin{align*}
\omega = \mu_{s} \mu_{s-1} \cdots \mu_{1}, \quad \mu_{i} \in \X_{p},  \quad \mu_{s} \prec_{\R} \ldots \prec_{\R} \mu_{1}.  
\end{align*}
Observe that $\ol(\mu_{1}) = \ol(m_{\omega})$,  and that  the word $\mu_{s}\mu_{s-1} \cdots \mu_{2}$  contains neither  $\ol(m_{\omega})$ nor $\ol(m_{\omega})^{*}$ when $\ol(m_{\omega}) \in D$.  Hence,  by Lemma \ref{lem:properties-alpha_R}, we obtain $\mu_{1} = \omega_{R} = \omega_{1}$. Applying the same argument recursively yields  $\mu_{i} = \omega_{i}$ for $i> 1$, completing the proof.
\end{proof}

Note that,  by Proposition \ref{prop:prime-words}, the factorization in Proposition  \ref{prop:unique-factorization} has the minimal possible length among all factorizations of a non-empty word into a product of prime words.

\section{Irreducible words with respect to the reduction order}\label{sec:3}

In this section, we study reducible words and irreducible words  defined with respect to the reduction order. 
Throughout, we fix a well-ordered set $X=O\cup D^{*},  D\subset O$ equipped with an order $\prec$, and we consider the induced reduction order $\prec_{\R}$ on the free monoid $\X$.

For  a non-zero polynomial $f\in \K\X$, the \textit{leading word} of $f$,  denoted  $\LW(f)$, is  the greatest word (with respect to  $\prec_{\R}$) appearing in $f$ with a nonzero coefficient. 

Let $I$ be an ideal of $\K\X$. A word $\alpha\in\X$ is called \textit{$I$-reducible} if $\LW(f)=\alpha$ for some polynomial $f\in I$. A word in $\X$ is called \textit{$I$-irreducible} if it is not \textit{$I$-reducible}. 

We denote
$$
X_{I} := \{x \in X \mid x \text{ is \textit{I}-irreducible} \}, \quad  
\X_{I}  := \{\omega \in \X \mid \omega \text{ is \textit{I}-irreducible} \}.
$$

\begin{example}
Let $(X, \prec)$ be a well-ordered set, and let $x\in D$ with its partner $x^{*}\in D^{*}$. Assume  there is an ideal $I$ of $\K\X$ such that
\begin{align*}
x^{*}x \in x + x^{*} +I, \quad xx^{*} \in x + x^{*} + I.
\end{align*} 
Recall  from  Example \ref{eg:x-x*}  that 
$$x^{*}x \succ_{\R} xx^{*} \succ_{\R} x^{*} \succ_{\R} x.$$
Thus,  we have 
$$\LW(xx^{*} - x - x^{*}) = xx^{*}\ \text{and}\ \LW(x^{*}x - x - x^{*}) = x^{*}x,$$
 so both  $xx^{*}$ and $x^{*}x$ are $I$-reducible.
\end{example}

\begin{example}\label{eg:xy^n}
Let $X = \{x,y\}$ with $x \prec y$, and let $I$ be the ideal of $\K\X$ generated by the relation $y-x^{2}$. Then,  we have: 
\begin{itemize}
    \item $y + I = x^{2} + I$ and $y \succ_{\R} x^{2}$.
    \item $yx + I = x^{3} + I = xy + I$ and  $yx \succ_{\R} x^{3} \succ_{\R} xy $.
    \item $yx^{2} + I = y^{2} + I =  x^{4} + I = xyx + I  = x^{2}y + I$ and  $yx^{2} \succ_{\R} y^{2} \succ_{\R} x^{4} \succ_{\R} xyx  \succ_{\R} x^{2}y.$
\end{itemize}
Hence,  $x, x^{2}, xy, x^{2}y$ are $I$-irreducible. It is not difficult to check that 
$$X_{I} = \{ x\}, \quad   \X_{I} = \{1, ~x, ~x^{2}, ~x^{n}y \mid n\ge 1 \}.$$
The residue classes of  $\X_{I}$ then  form a basis of the quotient algebra $\K\X/(y-x^{2})$.
\end{example}

We now show that every word can be expressed as a linear combination of irreducible words.
\begin{proposition}\label{prop:irrws-form-basis}
Let $I$ be an ideal of $\K\X$. Then the  residue classes of $I$-irreducible words $\{\omega + I \mid \omega\in \X_{I} \}$ form a basis of $\K\X/ I$.
\end{proposition}

\begin{proof}
We first show that $\{\omega + I \mid \omega\in \X_{I} \}$ is linearly independent. Suppose, to the contrary, that the  residue classes of $I$-irreducible words satisfy  a non-trivial linear relation. Then there exists a polynomial
\begin{align*}
p= \sum_{i=1}^{n} k_{i}\omega_{i} \in I,
\end{align*}
where $\omega_{1},\ldots,\omega_{n}$ are distinct $I$-irreducible words in $\X$. 
Reorder the terms so that  $\omega_{1} \succ_{\R} \ldots \succ_{\R} \omega_{n}$. Then $\LW(p)=\omega_{1}$, contradicting the fact that  $\omega_1$ is $I$-reducible. 

Next we show that every word in $\X$ is a linear combination of $I$-irreducible words. 
We proceed by induction on the reduction order $\prec_{\R}$. 

The empty word $1$ is clearly $I$-irreducible, so that the statement holds for $1$. 
Now let $\alpha$ be a $I$-reducible word. 
Then there is a polynomial $f\in I$ with $\LW(f)=\alpha$. Since 
$$\alpha + I = (\alpha -f) + I$$ 
and $\alpha -f$ is a linear combination of words strictly smaller than $\alpha$ under $\prec_{\R}$, the induction hypothesis implies that $\alpha + I$ is a linear combination of the residue classes of $I$-irreducible words.  This completes the proof.
\end{proof}

Next let $\K\X$ be the augmented free algebra on $X$ with the canonical augmentation $\epsilon:\K\X \rightarrow \K$ ($\epsilon(X)= 0$). 
We show that  prefixes of an irreducible words  are also irreducible.

\begin{proposition}\label{prop:prefix-irrws}
Let $(\K\X,\epsilon)$ be the augmented free algebra, and let  $I \subseteq \ker \epsilon$ be a proper ideal of $\K\X$. Then every prefix of an $I$-irreducible word in $\X$ is itself $I$-irreducible.
\end{proposition}

\begin{proof}
Let $\alpha$ be an $I$-irreducible word, and suppose $\alpha=\beta\gamma$ with  $\beta$ a prefix of $\alpha$.  Assume, to the contrary, that $\beta$ is $I$-reducible. 

By the augmentation $\epsilon$,  there exist words  $\eta \in \X^{+}$ and scalars $k_{\eta} \in \K$ such that 
\begin{align*}
\beta \in \sum_{\eta} k_{\eta} \eta + I, \qquad  \eta \prec_{\R} \beta.
\end{align*}
By Lemma \ref{lem:ac<bc}, multiplying by $\gamma$  preserves the inequality:
\begin{align*}
\alpha=\beta\gamma \in \sum_{\eta} k_{\eta} \eta\gamma + I, \qquad \eta\gamma \prec_{\R} \beta\gamma=\alpha.
\end{align*}
Hence, the polynomial 
$$p=\alpha - \sum_{\eta} k_{\eta} \eta\gamma \in I$$
satisfies $\LW(p) = \alpha$, contradicting  the assumption that $\alpha$ is $I$-irreducible. Hence, every prefix of $\alpha$ must be $I$-irreducible.
\end{proof}

\begin{remark}\label{rmk:suffix-irrws-false}
The analogue for suffixes of an $I$-irreducible word does not hold. For example,  in Example \ref{eg:xy^n}, each word  $x^{n}y$ ($n\geq 1$)  is $I$-irreducible, while its suffix $y$ is $I$-reducible. Thus an $I$-irreducible word may have reducible suffixes. 
\end{remark}

For a letter $y\in X$, set
$$
X^{\prec y} :=\{x\in X \mid x \prec y\}, \quad  X_{I}^{\prec y} :=\{x\in X_{I} \mid x \prec y \}.$$

We now record a basic property of irreducible letters in $O$.
\begin{lemma}\label{lem:irrl}
Let $(\K\X,\epsilon)$ be the augmented free algebra,  and let $I \subseteq \ker \epsilon$ be a proper ideal of $\K\X$.  If $y\in O$ is an $I$-irreducible letter, then 
\begin{align*}
y\notin X^{\prec y}\cdot \K\X + I.
\end{align*}
\end{lemma}

\begin{proof}
Note first that $X^{\prec y}\cdot \K\X + I\neq \K\X$,  since $1 \notin X^{\prec y}\cdot \K\X + I$ by the augmentation $\epsilon$.

Suppose that the claim fails. Then there exists an $I$-irreducible letter $y \in O$ such that 
$$y \in X^{\prec y}\cdot \K\X + I.$$
 Thus, we may write  
\begin{align*}
p=y - \sum_{i} k_{i}x_{i}f_{i} \in I, 
\end{align*}
where  each $x_{i}\in X^{\prec y}$, $f_{i}\in \K\X$,  and $k_{i} \in \K$.

Since $o(x_{i}f_{i}) = o(x_{i}) \preceq x_{i} \prec y = \ol(y)$ for all $i$,  every word appearing in  $\sum_{i} k_{i}x_{i}f_{i}$ is strictly smaller than $y$ w.r.t. $\prec_{\R}$. Consequently, $\LW(p)=y$,  contradicting the assumption that $y$ is $I$-irreducible. Hence, the claim follows.
\end{proof}

Given an augmented algebra $(A,\epsilon_{A})$ over $\K$. Choose a generating set $X$ of $A$ such that $X\subseteq \ker \epsilon_{A}$. 
By a slight abuse of notation, we also let $X$  generate a free algebra $\K\X$. Let  $\pi: \K\X \rightarrow A$  be  the canonical projection determined by the assignment map $f_{X}: X\rightarrow A$. Then $(\K\X,\epsilon)$ becomes an augmented algebra,  where its augmentation $\epsilon$ is induced by $\epsilon_{A}$ via the relation $\epsilon= \epsilon_{A}\circ \pi$. 
Let $I:=\ker\pi$. It is clear that $X, I \subseteq \ker \epsilon$.

Now assume $X = O \cup D^{*}$.  We will show that the corresponding set of irreducible letters is finite whenever $A$ is right  Noetherian.

\begin{lemma}\label{lem:Noetherian-irrls-finite}
Let $(A,\epsilon_{A})$ be an augmented algebra with a generating set $X=  O \cup D^{*}$, and let $\pi: \K\X \rightarrow A$ be the canonical projection and $I=\ker \pi$. 
If $A$ is right Noetherian, then the set $X_{I} \cap O$ is finite. Consequently, if the set $X_{I} \cap D^{*}$ is finite as well, then $X_{I}$ is finite.
\end{lemma}

\begin{proof}
 Assume, toward a contradiction,  that $|X_{I} \cap O|= +\infty$. 
 Choose distinct $I$-irreducible letters
 $$x_{1} \prec x_{2} \prec \ldots \prec x_{n} \prec \ldots$$
  in $X_{I} \cap O$, and consider the ascending chain of right ideals in $\K\X$:
\begin{align*}
x_{1} \cdot \K\X + I \subseteq \{x_{1},x_{2}\} \cdot \K\X + I \subseteq \ldots \subseteq \{x_{1},\ldots,x_{n}\} \cdot \K\X + I \subseteq \ldots,
\end{align*}

By the augmentation $\epsilon$, we have  $1 \notin \{x_{1},\ldots,x_{n}\} \cdot \K\X + I$, so each 
 right ideal in the chain is proper.  Lemma \ref{lem:irrl} implies that 
\begin{align*}
x_{n+1} \notin \{x_{1},\ldots,x_{n}\} \cdot \K\X + I, \qquad n\ge 1.
\end{align*} 
and hence  the above chain is a strictly ascending chain.

Since $\pi$ is an algebra map, applying $\pi$ yields a strictly ascending chain of right ideals of $A$:
\begin{align*}
\pi(x_{1}) \cdot  A  \subsetneq \{\pi(x_{1}),\pi(x_{2})\} \cdot A \subsetneq \ldots \subsetneq \{\pi(x_{1}),\ldots,\pi(x_{n})\} \cdot  A \subsetneq \ldots,
\end{align*}
contradicting the Noetherian property of $A$. Therefore, $X_{I} \cap O$ must be finite. The final assertion is immediate.
\end{proof}

\section{Comultiplications on words}\label{sec:4}

In this section, we use the reduction order, the reduction-factorization and prime words to study comultiplications on arbitrary words.

Suppose that $X$ is a set equipped with a map $t: X \rightarrow \mathbb{N}$. For $m \ge 0$, we write
\begin{align*}
X_{m} &:= \{x\in X  \mid t(x)=m\},\\
X_{+} &:= \{x\in X  \mid  t(x)\ge 1\}.
\end{align*}  
Then the free algebra $\K\X$ admits a grading 
$$\K\X=\bigoplus_{n= 0}^{\infty} \K\X_{n}$$
where $\deg(x) = t(x)$ for every letter $x\in X$. Note that $\K\X_{0}$ is not necessarily equal to the base field $\K$.

\begin{example}\label{example:filtration-t-degree}
Let $A$ be a filtered algebra over $\K$ with filtration $\{A_{n}\}_{n\ge 0}$. 
Choose a generating set $X$ of $A$ together with an assignment $f_{X}: X \rightarrow A$. Let $\pi: \K\X \rightarrow A$ be the canonical algebra map determined by $f_{X}$.
For each letter $x\in X$, define a map $t:X \rightarrow \mathbb{N}$ by 
\begin{align*}
t(x) = \min \{n \mid f_{X}(x)\in A_{n}\}.
\end{align*}
Then $\K\X=\bigoplus_{n= 0}^{\infty} \K\X_{n}$ becomes a graded algebra with $\deg(x) = t(x)$ for $x\in X$. 
The map $t$ extends naturally to  words by 
\begin{align*}
t(\omega) =  \min \{n \mid \pi(\omega)\in A_{n}\}, \quad \omega \in \X.
\end{align*}
Note that $t(xy) \le t(x) + t(y)$ for all $x,y\in X$.
Thus, for a word  $\omega = x_{1} \cdots x_{n}$, 
\begin{align*}
t(\omega) \le \sum_{i=1}^{n} t(x_{i}) = \sum_{i=1}^{n} \deg(x_{i}) = \deg(\omega).
\end{align*}
\end{example}

Assume throughout this section that $X = O \cup D^{*}$ is a well-ordered set with order $\prec$ and that we are given a map $t: X \rightarrow \mathbb{N}$ satisfying:
\begin{align*}
t(x^{*}) = t(x), \quad x \in D. 
\end{align*}
We now   introduce a comultiplication on the free algebra $\K\X$.
\begin{definition}
A \textit{comultiplication} on $\K\X$ is an algebra map $\Delta:\K\X \rightarrow \K\X \otimes \K\X$. 
We call  $\Delta$  \textit{skew-triangular} if for each $x\in X$, there exist elements
$$z_{x} \in X_{0}, ~x'\in \K\X^{+}, ~x'' \in X,$$ 
such that   
\begin{align*}
\Delta(x) &= x \otimes 1 + (1-z_{x}) \otimes x + \sum_{ t(x'') < t(x)}  x' \otimes x'', \quad x'' \prec x.
\end{align*}
For $x_{0} \in X_{0}$, this forces  $z_{x_{0}} = x_{0}$,  and hence 
$$\Delta(x_{0}) = x_{0} \otimes 1 + (1 - x_{0}) \otimes x_{0}.$$
\end{definition}

Note that a skew-triangular comultiplication is not necessarily coassociative.

For a word $\omega= x_{1}\cdots x_{n} \in \X$ with $x_{1},\ldots,x_{n}\in X$, we write
\[
z_{\omega} := \prod_{i=1}^{n}z_{x_{i}},\quad   
[1-z_{\omega}] := \prod_{i=1}^{n}(1-z_{x_{i}}) \in \K\X.
\]

\begin{lemma}\label{lem:deg+length}
Let $\Delta$ be a skew-triangular comultiplication on $\K\X$, and let $f\in \X^{+}$ satisfy
\begin{align*}
\Delta(f) = f \otimes 1 + [1-z_{f}] \otimes f + \sum_{} f' \otimes f'', \quad f'\in \K\X^{+}, ~f''\in \X^{+}.
\end{align*}
Then $\deg(f'') \le \deg(f)$, $|f''| \le |f|$ and $\deg(f'') + |f''| < \deg(f) + |f|$.
\end{lemma}
\begin{proof}
From the defining formula for $\Delta(x)$, we have  $|1|<|x''|=|x|$ for every $x\in X$.   Since $\Delta$ is an algebra map, this implies $|f''|\le |f|$. 

Next, for  $x\in X_{+}$, we have  $0=\deg(1) \le \deg(x'')< \deg(x)$, while $\deg(x)=0$ for $x\in X_{0}$. It follows that $\deg(f'')\le \deg(f)$, and equality $\deg(f'')= \deg(f)$ occurs only when $f=x_{1} \cdots x_{n}$ with $x_{i}\in X_{0}$ for all $i$.  In this case, 
$$\Delta(x_{i})= x_{i} \otimes 1 + (1-x_{i}) \otimes x_{i},$$
so in every nontrivial term one has  $|f''|<|f|$.  Hence, in all cases, we have $\deg(f'') + |f''| < \deg(f) + |f|$ for $f\in \X$.
\end{proof}

Now we determine the comultiplication of prime words.
\begin{lemma}\label{lem:coproduct-omega-1}
Suppose that $\Delta$ is a skew-triangular comultiplication on $\K\X$. If $\omega\in \X_{p}$, then there exist $\omega'\in \K\X^{+}$ and $\omega''\in \X^{+}$ such that
\begin{align*}
\Delta(\omega) = \omega \otimes 1 + [1-z_{\omega}] \otimes \omega + \sum_{}  \omega' \otimes \omega'',
\end{align*}
where each term in the sum satisfies
$$\omega'' \preceq_{\R} (\omega'')_{R} \prec_{\R} \omega,\ \text{and}\  |\omega''|\le |\omega|.$$
\end{lemma}

\begin{proof}
We proceed by induction on the length of words.  For $|\omega|=1$ the statement follows directly from the definition of  the  comultiplication on $X$. Assume now that $|\omega|\ge 2$ and write $\omega=\alpha x$ for $x\in X$. 
Since $\ol(\alpha) = \ol(m_{\alpha})$, 
Lemma \ref{lem:properties-alpha_R} (b) implies  $\alpha_{R}=\alpha$.   
Thus, by induction hypothesis (because $|\alpha|<|\omega|$),  we have
\begin{align*}
\Delta(\omega) =& \Delta(\alpha) \Delta(x) \\
=& (\alpha \otimes 1 + [1-z_{\alpha}] \otimes \alpha + \sum_{}  \alpha' \otimes \alpha'')(x \otimes 1  + (1-z_{x}) \otimes x + \sum_{}   x' \otimes x'') \\
=& \omega \otimes 1  + \alpha (1-z_{x}) \otimes x + \sum_{} \alpha  x' \otimes x''\\
& + [1-z_{\alpha}] x \otimes \alpha  + [1-z_{\omega}] \otimes \omega + \sum_{} [1-z_{\alpha}] x' \otimes \alpha x''\\
& + \sum_{}  \alpha' x \otimes \alpha''  + \sum_{}  \alpha' (1-z_{x}) \otimes \alpha'' x + \sum_{} \sum_{}  \alpha'  x' \otimes \alpha'' x'',
\end{align*}
where $x', \alpha'\in \K\X^{+}$, $x'' \in X$, $\alpha''\in \X^{+}$ and
\begin{align*}
x'' \prec x, \qquad \alpha'' \preceq_{\R} (\alpha'')_{R} \prec_{\R} \alpha, \qquad |\alpha''| \le |\alpha|.
\end{align*}

We now verify that every $\omega''$ arising in the above sums satisfies
$$\omega'' \preceq_{\R} (\omega'')_{R} \prec_{\R} \omega, \text{ and } |\omega''|\le |\omega|.$$
Since $\omega'' \preceq_{\R} (\omega'')_{R}$ and $\alpha, x\preceq_{\R}\omega_{R}=\omega$ by Lemma \ref{lem:properties-alpha_R} (d),  Lemma \ref{lem:prec_R-compatible-prec}, the definition of $\prec_{\R}$, and Lemma \ref{lem:ac<bc} apply in each of the following cases:

\begin{itemize}
\item $\omega'' = x$. $x_{R}=x \prec_{\R} \omega$.
\item $\omega'' = x''$. $(x'')_{R}=x'' \prec_{\R} x \prec_{\R} \omega$.
\item $\omega'' = \alpha$. $\alpha_{R} = \alpha \prec_{\R} \omega$.
\item $\omega'' = \alpha x''$.\\
$(\alpha x'')_{R} = \alpha_{R}x'' = \alpha x'' \prec_{\R} \alpha x = \omega$ if $\ol(m_{\alpha}) \succeq \ol(x'')$;\\
$(\alpha x'')_{R} =  x'' \prec_{\R} x \prec_{\R} \omega$ if $\ol(m_{\alpha}) \prec \ol(x'')$.
\item $\omega'' = \alpha''$. $(\alpha'')_{R} \prec_{\R} \alpha \prec_{\R} \omega$.
\item $\omega'' = \alpha''x$.\\
$(\alpha'' x)_{R} = (\alpha'')_{R}x \prec_{\R} \alpha x = \omega$ if $\ol(m_{\alpha''}) \succeq \ol(x)$; \\
$(\alpha'' x)_{R} =  x \prec_{\R} \omega$ if $\ol(m_{\alpha''}) \prec \ol(x)$.
\item $\omega'' = \alpha''x''$.\\
$(\alpha'' x'')_{R} = (\alpha'')_{R}x'' \prec_{\R} \alpha x'' \prec_{\R} \alpha x = \omega$ if $\ol(m_{\alpha''}) \succeq \ol(x'')$; \\
$(\alpha'' x'')_{R} = x'' \prec_{\R} x \prec_{\R} \omega$ if $\ol(m_{\alpha''}) \prec \ol(x'')$.
\item $|\alpha''|$, $|\alpha''x''|$, $|\alpha''x|$, $|\alpha|$, $|\alpha x''| \le |\alpha x|= |\omega|$.
\end{itemize}
Therefore, 
\begin{align*}
\Delta(\omega) = \omega \otimes 1 + [1-z_{\omega}] \otimes \omega + \sum_{} \omega' \otimes \omega'',
\end{align*}
where $\omega'\in \K\X^{+}$, $\omega''\in \X^{+}$, $\omega'' \preceq_{\R} (\omega'')_{R} \prec_{\R} \omega$ and $|\omega''|\le |\omega|$. This completes the proof. 
\end{proof}

Next we determine the comultiplication of  words in $\X \setminus  \X_{p}$.

\begin{lemma}\label{lem:coproduct-omega-2}
Suppose that $\Delta$ is a skew-triangular comultiplication on $\K\X$. Let $\omega \in \X \setminus \X_{p}$ and $\rf(\omega) = (\omega_{L}, \omega_{R})$. Then the following hold:
\begin{itemize}
\item [(a)] There exist $(\omega_{L})',(\omega_{R})' \in \K\X^{+}$ and words $(\omega_{L})'', (\omega_{R})'' \in \X^{+}$ such that
\begin{align*}
\Delta(\omega_{L}) = \omega_{L} \otimes 1 + [1-z_{\omega_{L}}] \otimes \omega_{L} + \sum_{} (\omega_{L})' \otimes (\omega_{L})'', \\
\Delta(\omega_{R}) = \omega_{R} \otimes 1 +  [1-z_{\omega_{R}}] \otimes \omega_{R} + \sum_{} (\omega_{R})' \otimes (\omega_{R})'', 
\end{align*}
where $\ol(m_{(\omega_{L})''}) \preceq \ol(m_{\omega_{L}}) \prec \ol(m_{\omega})$, $\ol(m_{(\omega_{R})''}) \preceq \ol(m_{\omega_{R}}) = \ol(m_{\omega})$, 
and $(\omega_{R})'' \preceq_{\R} ((\omega_{R})'')_{R} \prec_{\R} \omega_{R}$. 
\item [(b)] There exist $\omega'\in \K\X^{+}$ and $\omega''\in \X^{+}$ such that
\begin{align*}
\Delta(\omega) = \omega \otimes 1 + \omega_{L}[1-z_{\omega_{R}}] \otimes \omega_{R} +  [1-z_{\omega}] \otimes \omega + \sum_{} \omega' \otimes \omega'',
\end{align*}
where either $(\omega'')_{R} \prec_{\R} \omega_{R}$ or $\omega''  = f\omega_{R}$ for some $f\in \X^{+}$ satisfying $o(m_{f})\prec o(m_{\omega_{R}})$.
\end{itemize}
\end{lemma}

\begin{proof}
By Lemma \ref{lem:coproduct-omega-1}, we have $(\omega_{R})''\preceq_{\R} ((\omega_{R})'')_{R} \prec_{\R} \omega_{R}$.
Recall from Lemma \ref{lem:properties-alpha_R} that $\ol(m_{\omega_{L}})\prec \ol(m_{\omega})$ and $\ol(m_{\omega_{R}}) = \ol(m_{\omega})$. The comultiplication on $X$ then  implies that $\ol(m_{(\omega_{L})''}) \preceq \ol(m_{\omega_{L}}) \prec \ol(m_{\omega})$ and $\ol(m_{(\omega_{R})''}) \preceq \ol(m_{\omega_{R}}) = \ol(m_{\omega})$. This completes the proof of Part(a).  
Now, for Part(b), note that $\omega_{L} \neq 1$ since $\omega$ is not prime. Applying the result from Part(a), we obtain:
\begin{align*}
 &\Delta(\omega) \\
= &\Delta(\omega_{L})\Delta(\omega_{R})\\
= & (\omega_{L} \otimes 1 + [1-z_{\omega_{L}}] \otimes \omega_{L}+ \sum (\omega_{L})' \otimes (\omega_{L})'')(\omega_{R} \otimes 1 + [1-z_{\omega_{R}}] \otimes \omega_{R}+ \sum (\omega_{R})' \otimes (\omega_{R})'')\\
= & \omega \otimes 1 + \omega_{L}[1-z_{\omega_{R}}] \otimes \omega_{R} + \sum \omega_{L} (\omega_{R})' \otimes (\omega_{R})'' \\
& + [1-z_{\omega_{L}}]\omega_{R} \otimes \omega_{L} +  [1-z_{\omega}] \otimes \omega + \sum  [1-z_{\omega_{L}}](\omega_{R})' \otimes \omega_{L} (\omega_{R})''\\
& + \sum (\omega_{L})'\omega_{R} \otimes (\omega_{L})'' + \sum (\omega_{L})' [1-z_{\omega_{R}}] \otimes (\omega_{L})''\omega_{R} + \sum \sum (\omega_{L})'(\omega_{R})' \otimes (\omega_{L})'' (\omega_{R})''.
\end{align*}
Then, by the definition of $\prec_{\R}$ and Lemmas \ref{lem:properties-alpha_R} and \ref{lem:prec_R-compatible-prec}, we conclude the following:
\begin{itemize}
\item 
$\omega'' = (\omega_{R})''$.
$((\omega_{R})'')_{R} \prec_{\R} \omega_{R}$.
\item 
$\omega'' = \omega_{L}$.
$(\omega_{L})_{R} \prec_{\R} \ol(m_{\omega}) \preceq_{\R} m_{\omega} \preceq_{\R} \omega_{R}$.
\item 
$\omega'' = \omega_{L}(\omega_{R})''$.    
\subitem  
$(\omega_{L}(\omega_{R})'')_{R} \prec_{\R}  \ol(m_{\omega}) \preceq_{\R} m_{\omega} \preceq_{\R} \omega_{R}$ if $\ol(m_{(\omega_{R})''}) \preceq \ol(m_{\omega_{L}})$;
\subitem  
$(\omega_{L}(\omega_{R})'')_{R} = ((\omega_{R})'')_{R}  \prec_{\R} \omega_{R}$ if $\ol(m_{(\omega_{R})''}) \succ \ol(m_{\omega_{L}})$.
\item 
$\omega'' = (\omega_{L})''$. $((\omega_{L})'')_{R}  \prec_{\R} \ol(m_{\omega})  \preceq_{\R}   \omega_{R}$. 
\item 
$\omega'' = (\omega_{L})''\omega_{R}$.    
\subitem 
$\ol(m_{(\omega_{L})''}) \prec_{\R} \ol(m_{\omega}) = \ol(m_{\omega_{R}})$ and $((\omega_{L})''\omega_{R})_{R}= \omega_{R}$.
\item 
$\omega'' =(\omega_{L})'' (\omega_{R})''$.
\subitem 
$((\omega_{L})'' (\omega_{R})'')_{R} \prec_{\R} \ol(m_{\omega})  \preceq_{\R} \omega_{R}$ if $\ol(m_{(\omega_{R})''}) \preceq \ol(m_{(\omega_{L})''})$; 
\subitem  
$((\omega_{L})'' (\omega_{R})'')_{R} =  ((\omega_{R})'')_{R} \prec_{\R} \omega_{R}$ if $\ol(m_{(\omega_{R})''}) \succ \ol(m_{(\omega_{L})''})$.
\end{itemize}
Therefore, Part(b) holds.
\end{proof}

\section{Words and Irreducible letters}\label{sec:5}
This section applies the comultiplication constructed in Section~\ref{sec:4} to show that every word can be written as a linear combination of products of irreducible letters.

Throughout this section, assume that $X = O \cup D^{*}$ ($D\subseteq O$) is a well-ordered set with respect to $\prec$, equipped with a map $t : X \to \mathbb{N}$ satisfying $t(x^{*}) = t(x)$ for all $x \in D$. Let $\K\X$ denote the augmented free algebra on $X$ with the canonical augmentation $\epsilon : \K\X \to \K$ (that is, $\epsilon(X) = 0$). We begin with a key reduction lemma for prime words. This result plays a central role in establishing Lemma~\ref{lem:rl-sum-smaller-irrls} and Theorem~\ref{thm:rl-sum-smaller-irrls}.

\begin{lemma}\label{lem:wr-sum-smaller-words}
Let $(\K\X,\epsilon)$ be the augmented free algebra, and let $I \subseteq \ker \epsilon$ be a proper ideal of $\K\X$.
Suppose that $\Delta$ is a skew-triangular comultiplication on $\K\X$ satisfying $\Delta(I) \subseteq I \otimes \K\X + \K\X \otimes I$, and that for every $x\in X_{0}$, there exists $c_{z}\in \K\X$ such that 
$$c_{z}(1-z), (1-z)c_{z}\in 1+ I\ \text{and}\ \ol(m_{c_{z}}) \preceq \ol(z).$$
Let $\omega\in \X^{+}$,  and suppose there exist $f,g\in \X^{+}$ and $k_{f},k_{g}\in \K$ such that
\begin{align}\label{formula:wr-fwr-g-I}
p_{\omega} = \omega_{R} + \sum_{f} k_{f}f\omega_{R} + \sum_{g} k_{g} g\in I, \quad \ol(m_{f}) \prec \ol(m_{\omega_{R}}), ~g_{R}\prec_{\R} \omega_{R}.
\end{align} 
Then there exist words $h\in \X^{+}$ and coefficients $k_{h}\in \K$ such that 
\begin{align}\label{formula:wr-h-I}
\omega_{R} \in \sum_{h} k_{h} h + I, \quad h_{R} \prec_{\R} \omega_{R}.
\end{align}
\end{lemma}

\begin{proof}
Let 
\begin{align*}
W_{1}:= \{f \mid p_{\omega} \in I \text{ and } k_{f}\neq 0\}.
\end{align*}
If $W_{1}=\emptyset$, then it is done. Let $W_{1}\neq \emptyset$. By Lemmas \ref{lem:coproduct-omega-2}, there exist $(\omega_{R})', f',g' \in \K\X^{+}$ and words $(\omega_{R})'', f'',g''\in \X^{+}$ such that
\begin{align*}
& \Delta(\omega_{R} + \sum_{f\in W_{1}} k_{f} f\omega_{R} + \sum_{g} k_{g} g) \\
= & (\omega_{R} \otimes 1 + [1-z_{\omega_{R}}] \otimes \omega_{R} +  \sum_{} (\omega_{R})' \otimes (\omega_{R})'') \\
& + \sum_{f\in W_{1}} k_{f}(f\otimes 1 + [1-z_{f}] \otimes f + \sum_{} f' \otimes f'')(\omega_{R} \otimes 1 + [1-z_{\omega_{R}}] \otimes \omega_{R} + \sum_{} (\omega_{R})' \otimes (\omega_{R})'')\\
& + \sum_{g} k_{g} (g\otimes 1  + [1-z_{g}] \otimes g + \sum_{} g' \otimes g'')
\end{align*}
\begin{align*}
= & \omega_{R} \otimes 1 + [1-z_{\omega_{R}}] \otimes \omega_{R} + \sum_{} (\omega_{R})' \otimes (\omega_{R})'' \\
& + \sum_{f\in W_{1}} k_{f} (f \omega_{R} \otimes 1  + f [1-z_{\omega_{R}}] \otimes \omega_{R} + \sum_{} f (\omega_{R})' \otimes (\omega_{R})''\\
& +  [1-z_{f}] \omega_{R} \otimes f  + [1-z_{f}] [1-z_{\omega_{R}}] \otimes f \omega_{R} + \sum_{} [1-z_{f}] (\omega_{R})' \otimes f (\omega_{R})''\\
& + \sum_{} f' \omega_{R} \otimes f''  +  \sum_{} f' [1-z_{\omega_{R}}] \otimes f'' \omega_{R} +  \sum_{}  \sum_{} f' (\omega_{R})' \otimes f'' (\omega_{R})'')\\
& + \sum_{g} k_{g} (g\otimes 1 +  [1-z_{g}] \otimes g + \sum_{} g' \otimes g''),
\end{align*}
where  $\deg(f'') + |f''| < \deg(f) + |f|$ by Lemma \ref{lem:deg+length}.

Define 
$$p_{0}:= [1-z_{\omega_{R}}],\ \  p_{1}:= (1+\sum_{f\in W_{1}} k_{f}f)[1-z_{\omega_{R}}],\ \ g_{f}:= [1-z_{f}][1-z_{\omega_{R}}] \ \text{for}\ f\in W_{1}.$$
Since $\epsilon(p_{1})=1$ and $\epsilon(I)=0$,  the subspace
 $\K p_{1} + I$ is a direct sum.
 
 Now fix an element $f_{0}\in W_{1}$ and define
\begin{align*}
W_{f_{0}}:=\{f\in W_{1} \mid g_{f}\in g_{f_{0}} + I\}.
\end{align*}
Note that $W_{f_{0}}\neq \emptyset$ since $f_{0}\in W_{f_{0}}$.

We now consider the following cases:

\noindent
\textit{Case 1:} $g_{f_{0}} \in \K p_{1} \bigoplus I$.   Since $\epsilon(g_{f_{0}}) =\epsilon(p_{1})=1$ and $\epsilon(v)=0$,  we may write $g_{f_{0}} = p_{1} + v$ for some $v\in I$.  This implies 
 $$(1 + \sum_{f\in W_{1}} k_{f}f) \in [1-z_{f_{0}}] + I.$$ 
Substituting into equation (\ref{formula:wr-fwr-g-I})  yields:
\begin{align*}
[1-z_{f_{0}}]\omega_{R} + \sum_{g} k_{g} g\in I.
\end{align*}
Recall that,  for every $x\in X_{0}$, $c_{x}(1-x), (1-x)c_{x}\in 1+ I$ satisfying $\ol(m_{c_{x}}) \preceq \ol(x)$. 
Since $\ol(m_{z_{f_{0}}}) \preceq \ol(m_{f_{0}}) \prec \ol(m_{\omega_{R}})$, it is not difficult to show that there exist words $g_{1}\in \X^{+}$ and coefficients $k_{g_{1}} \in \K$ such that
\begin{align*}
\omega_{R} + \sum_{g_{1}} k_{g_{1}} g_{1}\in I,  \qquad (g_{1})_{R} \prec_{\R} \omega_{R}.
\end{align*}

\noindent
\textit{Case 2:} $g_{f_{0}} \notin \K p_{1} \bigoplus I$. In this case, the sum $\K p_{1} + I + \K g_{f_{0}}$ is direct. Choose a subspace $U\subseteq \K\X$ such that $\K\X = \K p_{1} \bigoplus I \bigoplus \K g_{f_{0}} \bigoplus U$. Define
 a linear map $\phi:\K\X \rightarrow \K$ by
\begin{align*}
\phi(p_{1}) =1, \quad \phi(g_{f_{0}})= 0, \quad \phi(I)=\phi(U)=0. 
\end{align*}
Let $W_{1}' := W_{1} \setminus W_{f_{0}}$.  We record the following observations:
\begin{itemize}
    \item $((\omega_{R})'')_{R} \prec_{\R} \omega_{R}$ by Lemma \ref{lem:coproduct-omega-2} (a) or Lemma \ref{lem:coproduct-omega-1};
    \item $f_{R}, (f(\omega_{R})'')_{R}, (f'')_{R}, (f''(\omega_{R})'')_{R} \prec_{\R} \omega_{R}$, because $\ol(m_{f''}) \preceq \ol(m_{f}) \prec \ol(m_{\omega_{R}})$ and $((\omega_{R})'')_{R} \prec_{\R} \omega_{R}$.
    \item $(g'')_{R} \preceq_{\R} g_{R} \prec_{\R} \omega_{R}$ by Lemma \ref{lem:coproduct-omega-2} (b).
\end{itemize}
Apply the linear map $\phi \otimes \id$ to 
\begin{align*}
\Delta(\omega_{R} + \sum_{f\in W_{1}} k_{f} f\omega_{R} + \sum_{g} k_{g} g) \in I \otimes \K\X + \K\X \otimes I.
\end{align*}
Then there exist words $g_{2} \in \X^{+}$ and coefficients $k_{g_{2}} \in \K$ such that
\begin{align}\label{formula:5.1-3}
\omega_{R} + \sum_{f\in W_{1}'} k_{f} \phi(g_{f}) f\omega_{R} + \sum_{f\in W_{1}} k_{f} \sum_{}\phi( f'p_{0}) f''\omega_{R} + \sum k_{g_{2}} g_{2} \in I,
\end{align}
where 
$$\ol(m_{f''}) \prec \ol(m_{\omega_{R}}),\ \deg(f'') + |f''| < \deg(f) + |f|\ \text{and}\ (g_{2})_{R} \prec_{\R} \omega_{R}.$$

Note that $|W_{1}'| < |W_{1}|$. By repeating the above process, we may eliminate the second term in (\ref{formula:5.1-3}). Consequently,  there exit words $q_{f}, g_{3}\in \X^{+}$ and coefficients $k_{q_{f}},k_{g_{3}}\in \K$ such that
\begin{align}\label{formula:wr-qf-I}
\omega_{R} + \sum_{f\in W_{1}} \sum_{q_{f}} k_{q_{f}} q_{f} \omega_{R} + \sum k_{g_{3}} g_{3} \in I,
\end{align}
where $(g_{3})_{R} \prec_{\R} \omega_{R}$, $\ol(m_{q_{f}}) \prec \ol(m_{\omega_{R}})$,  and  $\deg(q_{f})+ |q_{f}| < \deg(f) + |f|$ by Lemma \ref{lem:deg+length}. 

Define
\begin{align*}
W_{2} &:= \{q_{f} \mid q_{f} \text{ appears in } (4) \}, \\
d_{1} &:= \max\{\deg(f) + |f| \mid f \in W_{1} \},\\
d_{2} &:= \max\{\deg(w) + |w| \mid  w \in W_{2} \}
\end{align*}
We now prove (\ref{formula:wr-h-I}) by induction on $d_{1}$. 

If $d_{1}=1$, then $\deg(f)=0$ and $|f|=1$ for all $f\in W_{1}$.  Since $\Delta(f)= f\otimes 1  + (1-f) \otimes f$, we have $f'=0$ for all $f\in W_{1}$. 
Hence,  $k_{q_{f}}=0$,  and (\ref{formula:wr-qf-I}) takes the desired form:
\begin{align*}
\omega_{R} + \sum k_{g_{3}} g_{3} \in I, \qquad (g_{3})_{R} \prec_{\R} \omega_{R}. 
\end{align*}
Let $d_{1} \ge 1$. By construction,
$$\deg(q_{f})+ |q_{f}| < \deg(f) + |f|, \ \text{for all}\ q_{f} \in W_{2},$$
and hence, $d_2<d_1$.   Thus the induction hypothesis applies to $W_2$, and yields
the desired form (\ref{formula:wr-h-I}) for $\omega_{R}$. 
\end{proof}

Next we show that every reducible letter can be represented by a linear combination of products of letters with strictly lower reduction order.

\begin{lemma}\label{lem:rl-sum-smaller-irrls}
Let $(\K\X,\epsilon)$ be the augmented free algebra, and let $I \subseteq \ker \epsilon$ be a proper ideal of $\K\X$. Suppose $\Delta$ is a skew-triangular comultiplication on $\K\X$ such that $\Delta(I) \subseteq I \otimes \K\X + \K\X \otimes I$.
Assume that for every $z\in X_{0}$, there exists $c_{z}\in \K\X$ such that 
$$c_{z}(1-z), (1-z)c_{z}\in 1+ I\ \text{and}\ \ol(m_{c_{z}}) \preceq \ol(z).$$
If $x\in X$ is an $I$-reducible letter,
then there exist letters $x_{1},\ldots,x_{n}$ in $X$ and coefficients $k_{x_{1}\cdots x_{n}}\in \K$ such that
\begin{align*}
x \in \sum_{} k_{x_{1}\cdots x_{n}} x_{1} \cdots x_{n} + I, \qquad x_{1}, \ldots, x_{n} \prec x.
\end{align*}
That is, $x$ can be expressed modulo $I$ as a linear combination of monomials in strictly smaller letters.
\end{lemma}

\begin{proof}
If $x \in I$, the result is immediate. Assume therefore that $x \notin I$. We proceed by case analysis based on the type of $x$.

\noindent
\textit{Case 1}: $x\in O$.
By the augmentation $\epsilon$ and Corollary \ref{cor:prec_R-compatible-prec-2}, there exist $\omega \in \X^{+}$ and $k_{\omega} \in \K^{*}$ such that 
\begin{align*}
x\in \sum_{} k_{\omega} \omega + I, \quad \ol(\omega)\prec x.
\end{align*}
Define the sets:
\begin{align*}
W & :=\{\omega \mid x\in \sum_{} k_{\omega} \omega + I\},\\ 
W_{R}& :=\{\omega_{R} \mid \omega\in W\},
\end{align*}
and  let $\beta$ be  the greatest word  in $W_{R}$ (w.r.t. $\prec_{\R}$).  We can then  express $x$ as:
\begin{align*}
x\in \sum_{\alpha\beta\in W} k_{\alpha\beta} \alpha\beta + \sum_{\omega\in W \atop {\omega_{R} \prec_{\R} \beta}} k_{\omega} \omega + I, \qquad \ol(\alpha\beta), \ol(\omega) \prec x.
\end{align*}
Without loss of generality, we assume $\sum_{\alpha\beta\in W} k_{\alpha\beta} \alpha\beta\notin I$.

Since $(\X, \prec_{\R})$ is well-ordered by Lemma \ref{lem:reduction-order-well},  we proceed by induction  on $\beta$ w.r.t. $\prec_{\R}$. 
Note that $\beta \in \X^{+}$.

Let $a$ be the least non-empty word in $\X$ (i.e., the least letter of $X$ w.r.t. the order $\prec$ by Lemma \ref{lem:x-least} (d)).
If $\beta = a$, then  for each  $\omega\in W$, we have $\omega = a^{n(\omega)}$ for some $n(\omega)\ge 1$, where $a = o(\omega) \prec x$. Thus the claim holds in this case.
Assume $\beta \succ_{\R} a$. 
If $\beta \prec_{\R} x$, then by Lemma \ref{lem:properties-alpha_R}(d), for every $\omega \in W$ and every letter $b$ in $\omega$, we have $b \preceq_{\R} \omega_{R} \preceq_{\R} \beta \prec_{\R} x$. Lemma \ref{lem:prec_R-compatible-prec} then implies $b \prec x$, and again the claim holds. 

Now suppose $\beta \succeq_{\R} x$ (i.e., $l(\beta), \ol(\beta) \succeq x$ by Corollary \ref{cor:prec_R-compatible-prec-2}). Then $\alpha \neq 1$ since $\ol(\alpha\beta) \prec x$. 

By the comultiplication structure and Lemma \ref{lem:coproduct-omega-2}, there exist $z_{x} \in X_{0}$, $x', (\alpha\beta)', \omega' \in \K\X^{+}$, $x'' \in X$ and $(\alpha\beta)'', \omega'' \in \X^{+}$ such that:
\begin{align*}
& \Delta( \sum_{\alpha\beta\in W} k_{\alpha\beta} \alpha\beta + \sum_{\omega\in W \atop {\omega_{R} \prec_{\R} \beta}} k_{\omega} \omega - x) \\
=&\sum_{\alpha\beta\in W} k_{\alpha\beta}(\alpha\beta \otimes 1 + \alpha [1-z_{\beta}] \otimes \beta 
+ [1-z_{\alpha\beta}] \otimes \alpha\beta + \sum_{} (\alpha\beta)' \otimes (\alpha\beta)'') \\
& + \sum_{\omega \in W \atop{\omega_{R}\prec_{\R} \beta}} k_{\omega}(\omega \otimes 1  + [1-z_{\omega}] \otimes \omega + \sum_{} \omega' \otimes \omega'')\\
& - (x \otimes 1 +  (1-z_{x}) \otimes x + \sum_{} x' \otimes x''),
\end{align*}
where:
\begin{itemize}
\item $x'' \prec x$, i.e., $x '' \prec_{\R} x$,
\item $(\omega'')_{R} \preceq_{\R} \omega_{R} \prec_{\R} \beta$,
\item either $((\alpha\beta)'')_{R} \prec_{\R} \beta$ or $(\alpha\beta)'' = u\beta$ for some $u \in \X^{+}$ with $o(m_{u}) \prec o(m_{\beta})$.
\end{itemize}
Let $p := \sum_{\alpha\beta \in W} k_{\alpha\beta} \alpha[1 - z_{\beta}]$. Note that $p \notin I$ since $\sum_{\alpha\beta \in W} k_{\alpha\beta} \alpha\beta \notin I$. Therefore, $\K p + I$ is a direct sum. Since $\alpha \in \X^{+}$ for every $\alpha\beta \in W$, we also have that $\K p + I + \K 1$ is a direct sum by the augmentation $\epsilon$.

Let $\K\X = \K p \oplus I \oplus \K 1 \oplus V$ for some subspace $V \subseteq \K\X$. Define a linear map $\varphi : \K\X \rightarrow \K$ by:
\begin{align*}
\varphi(p) =1, \quad \varphi(1) = 0, \quad \varphi(I) = \varphi(V)=0. 
\end{align*}

Since
\begin{align*}
\Delta( \sum_{\alpha\beta\in W} k_{\alpha\beta} \alpha\beta + \sum_{\omega\in W \atop {\omega_{R} \prec_{\R} \beta}} k_{\omega} \omega - x) 
\in  I \otimes \K\X + \K\X \otimes I,
\end{align*}
applying $\varphi \otimes \id$ yields that there exist $f, g \in \X^{+}$ and $k_{f}, k_{g} \in \K$ such that:
\begin{align*}
\beta +   \sum_{f} k_{f} f\beta + \sum_{g} k_{g} g\in I, \qquad  \ol(m_{f}) \prec \ol(m_{\beta}), \quad g_{R}\prec_{\R} \beta.
\end{align*}
By Lemma \ref{lem:wr-sum-smaller-words}, there exist words $h\in \X^{+}$ and coefficients $k_{h}\in \K$ such that
\begin{align*}
\beta \in \sum_{h} k_{h} h + I, \qquad  h_{R} \prec_{\R} \beta.
\end{align*} 
We can now rewrite $x$ as: 
\begin{align*}
x\in \sum_{\alpha\beta\in W} \sum_{h} k_{\alpha\beta}k_{h} \alpha h + \sum_{\omega\in W \atop {\omega_{R} \prec_{\R} \beta}} k_{\omega} \omega + I, \qquad h_{R} \prec_{\R} \beta.
\end{align*}
Observe that $(\alpha h)_{R} \prec_{\R} \beta$, since 
\begin{itemize}
\item if $\ol(m_{\alpha}) \prec \ol(m_{h})$, then $(\alpha h)_{R} = h_{R} \prec_{\R} \beta$;
\item if $\ol(m_{\alpha}) \succeq \ol(m_{h})$, then $(\alpha h)_{R} \prec_{\R} \ol(m_{\beta}) \preceq_{\R} m_{\beta} \preceq_{\R} \beta$.
\end{itemize}
By the induction hypothesis (since $(\alpha h)_{R}, \omega_{R} \prec_{\R} \beta$), there exist letters $x_{1}, \ldots, x_{n} \in X$ and coefficients $k_{x_{1}\cdots x_{n}} \in \K$ such that:
\begin{align*}
x \in \sum_{} k_{x_{1}\cdots x_{n}} x_{1} \cdots x_{n} + I, \qquad x_{1}, \ldots x_{n} \prec x.
\end{align*}

\noindent
\textit{Case 2}: $x^{*}\in D^{*}$ for an original element $x \in D$.  

Recall that $x^{*}$ is the second least element among words $w \in \X$ satisfying $\ol(w) = x$ by Lemma \ref{lem:x-least} (c). 
It follows from Corollary \ref{cor:prec_R-compatible-prec-2} that there exist $\omega \in \X^{+}$ and $k_{x}, k_{\omega} \in \K$, such that
\begin{align*}
x^{*} - k_{x} x \in \sum_{} k_{\omega} \omega + I, \qquad o(\omega)\prec x.
\end{align*}
Recall that $t(x^{*})=t(x)$. The remainder of the proof follows the same patten of Case 1,
using $\Delta(x^{*})$ instead of $\Delta(x)$. We conclude that there exist letters $y_{1}, \ldots, y_{m} \in X$ and coefficients $k_{y_{1}\cdots y_{m}} \in \K$ such that:
\begin{align*}
x^{*} \in k_{x} x +  \sum_{} k_{y_{1}\cdots y_{m}} y_{1} \cdots y_{m} + I, \qquad y_{1},\ldots,y_{m} \prec x \prec x^{*}.
\end{align*}
\end{proof}

We now prove that every letter or non-empty word can be expressed as a linear combination of products of $I$-irreducible letters w.r.t. the reduction order. In particular, this shows that the images $\pi(X_I)$ generate the quotient algebra $\K\X/I$.

\begin{theorem}\label{thm:rl-sum-smaller-irrls}
Let $(\K\X,\epsilon)$ be the augmented free algebra, and let $I \subseteq \ker \epsilon$ be a proper ideal of $\K\X$.
Suppose $\Delta$ is a skew-triangular comultiplication on $\K\X$ such that $\Delta(I) \subseteq I \otimes \K\X + \K\X \otimes I$.
Assume that for every $z \in X_{0}$, there exists $c_{z} \in \K\X$ satisfying:
\begin{itemize}
\item $c_{z}(1-z), (1-z)c_{z} \in 1 + I\ (c_z$ is the inverse of $1-z$ modulo $I)$,
\item $\ol(m_{c_{z}}) \preceq \ol(z)$ $($order control condition$)$.
\end{itemize}
If $x \in X$ is an $I$-reducible letter, then there exist $I$-irreducible letters $x_{1}, \ldots, x_{n} \in X$ and scalars $k_{x_{1}\cdots x_{n}} \in \K$ such that
\begin{align*}
x \in \sum_{} k_{x_{1}\cdots x_{n}}x_{1} \cdots x_{n} + I, \qquad x_{1},\ldots,x_{n} \prec x.
\end{align*}
Consequently, every non-empty word in $\X$ can be expressed modulo $I$ as a linear combination of products of $I$-irreducible letters.
\end{theorem}

\begin{proof}
We proceed by induction on reducible letters in $X$ with respect to the well-ordering $\prec$ (since $(X, \prec)$ is well-ordered).

By Lemma \ref{lem:rl-sum-smaller-irrls}, there exist letters $x_{1}, \ldots, x_{n} \in X$ and scalars $k_{x_{1}\cdots x_{n}} \in \K$ such that
\begin{align*}
x \in \sum_{} k_{x_{1}\cdots x_{n}} x_{1} \cdots x_{n} + I,  \qquad  x_{1},\ldots,x_{n} \prec x.
\end{align*}
Now, by the induction hypothesis (since $x_{1}, \ldots, x_{n} \prec x$), each $x_i$ that is $I$-reducible can be expressed modulo $I$ as a linear combination of products of $I$-irreducible letters. Substituting these expressions yields the existence of $I$-irreducible letters $y_{1}, \ldots, y_{m} \in X_{I}$ and scalars $k_{y_{1}\cdots y_{m}} \in \K$ such that
\begin{align*}
x \in \sum_{} k_{y_{1}\cdots y_{m}} y_{1} \cdots y_{m} + I 
\qquad y_{1},\ldots,y_{m} \prec x.
\end{align*}
Therefore, every reducible letter in $X$ can be represented modulo $I$ as a linear combination of products of $I$-irreducible letters. Since irreducible letters trivially satisfy this condition, it follows that all letters in $X$ have this property. Consequently, every non-empty word in $\X$ can be expressed modulo $I$ as a linear combination of products of $I$-irreducible letters.
\end{proof}

\section{Application to  Noetherian pointed Hopf algebras}\label{sec:6}

In this section, we apply the results obtained in Section 5 to pointed Hopf algebras.  We show that every right or left Noetherian pointed Hopf algebra is affine.  

Let $H$ be a pointed Hopf algebra over $\K$, $\G(H)$ the group of its group-like elements.  For any $g\in \G(H)$,  recall that 
\begin{align*}
\Delta_{H}(1-g) = (1-g) \otimes 1 + g \otimes (1-g), \quad \epsilon_{H}(1-g)=0,
\end{align*}
that is, $1-g$ is a $(g,1)$-skew primitive element. In what follows, we will take $1-g$ as a generator of $H$ for every $g\in G(H)$, rather than $g$ itself. The choice is motivated by the fact that  because nontrivial group-like elements pose significant challenges in defining the reduction order and irreducible words when used as generators.

Recall that the coradical filtration $\{H_{(n)}\}_{n\ge 0}$ of $H$ is a Hopf algebra filtration of $H$. By Lemma \ref{lem:corad-pointed-Hopfalg}, we may choose a generating set $X$ of $H$, along with the assignment $f_{X}: X \rightarrow H$ and a map $t: X \rightarrow \mathbb{N}$ defined by 
$$t(x) = \min\{n \mid f_{X}(x) \in H_{(n)}\}$$
 (see Example \ref{example:filtration-t-degree}), such that for $x\in X_{0}$, we have $\pi(1-x)\in \G(H)$. Here,  $\pi: \K\X \rightarrow H$ is the canonical projection induced by $f_{X}$, and $\K\X$ is the free algebra on $X$.

Let $I=\ker \pi$.  Recall that  for  all $x\in X_{0}$,  $\pi(1-x) \in \G(H)$. 
Suppose  $\ord(\pi(1-x)) = n < \infty$. Then 
\begin{align}\label{formula:invertible-relation-finite-1}
(1-x)^{n} \in 1 + I,
\end{align}
or equivalently,
\begin{align}\label{formula:invertible-relation-finite-2}
x^{n} \in a_{n-1}x^{n-1} + \ldots + a_{1}x + I, \quad \text{ for some } a_{n-1},\ldots,a_{1} \in \Z.
\end{align}
If instead $\ord(\pi(1-x))) = \infty$, then there exists a letter $x^{*} \in X_{0}$ such that
\begin{align}\label{formula:invertible-relation-infinite-1}
& (1-x)(1-x^{*})\in 1+ I, \qquad (1-x^{*})(1-x)\in 1 + I, 
\end{align}
or equivalently, 
\begin{align}\label{formula:invertible-relation-infinite-2}
xx^{*} \in x + x^{*} + I, \qquad x^{*}x \in x + x^{*} +I.
\end{align}

For such elements $x,x^{*}\in X_{0}$, we have the following result.
\begin{lemma}\label{lem:relation-x-x*}
Let $(\K\X,\epsilon)$ be the augmented free algebra. Let $I \subseteq \ker \epsilon$ be a proper ideal of $\K\X$, $x \in D$ and $x^{*} \in D^{*}$. Suppose that 
$$xx^{*} \in x + x^{*} + I\ \text{and}\ x^{*}x \in x + x^{*} +I.$$
 Then the following hold:
\begin{itemize}
    \item [(a)] If $x$  is $I$-reducible, then $x^{*}$ is $I$-reducible.
    \item [(b)] If $x^{*}$  is $I$-reducible, then there exist $\omega \in \X^{+}$ and $k_{x}, k_{\omega} \in \K$ such that
\begin{align}\label{formula:x*}
x^{*} \in k_{x} x + \sum_{} k_{\omega} \omega + I, \qquad  o(\omega)\prec x.
\end{align}
Moreover,
\begin{itemize}
\item  if $k_{x} = 0$, then $x$ is $I$-reducible;
\item  if $k_{x} \neq 0$, then $x^{2}$ is $I$-reducible.
\end{itemize}
\end{itemize}
\end{lemma}

\begin{proof}
Assume $x$ is $I$-reducible. By Corollary \ref{cor:prec_R-compatible-prec-2}, we have $x\in \sum_{} k_{\omega} \omega + I$ for some words $\omega \in \X^{+}$ satisfying $o(\omega)\prec x$, and $k_{\omega} \in \K$.
Since $xx^{*} \in x + x^{*} + I$, there is a polynomial 
\begin{align*}
f= x^{*} + x - \sum_{} k_{\omega} \omega x^{*} \in I, \qquad o(\omega x^{*})\prec x.
\end{align*}
Thus,  $\LW(f) = x^{*}$, and hence $x^{*}$ is $I$-reducible.

We now prove Part(b). By Lemma \ref{lem:x-least} (c) and Corollary \ref{cor:prec_R-compatible-prec-2},   $x^{*}$ admits a representation of the form (\ref{formula:x*}). 
Substituting this into the relation: $x^{*}x \in x + x^{*} +I$, yields the polynomial
\begin{align*}
p = (k_{x} + 1) x + \sum_{} k_{\omega} \omega - k_{x} x^{2} -  \sum_{} k_{\omega} \omega x \in I, \qquad \ol(\omega) = \ol(\omega x) \prec x.
\end{align*}
The leading word of $p$ is then easily determined:  $\LW(p) = x$ if $k_{x} = 0$; and $\LW(p) = x^{2}$ if $k_{x} \neq 0$.
\end{proof}

We now demonstrate that for a pointed Hopf algebra, the residue classes of its irreducible letters (with respect to a suitably chosen generating set) form a generating set. It follows immediately that any right or left Noetherian pointed Hopf algebra is affine.

\begin{theorem}\label{thm:noetherian-pointed-affine}
Let $H$ be a pointed Hopf algebra over $\K$. The following hold:
\begin{itemize}
    \item [(a)] There exits a generating set $X$ of $H$ such that $\pi(X_{I})$ generates $H$ as an algebra, where $\pi: \K\X \rightarrow H$ is the canonical projection and  $I=\ker\pi$.
    \item [(b)] If $H$ is right or left Noetherian, then $H$ is affine.
\end{itemize}
\end{theorem}

\begin{proof}
By Lemmas \ref{lem:corad-pointed-Hopfalg} and \ref{lem:coproduct-pointed-coalgebra} 
, we see that for $0\neq v = 1-u \in H_{(0)}$ with $u\in \G(H)$, 
\begin{align*}
\Delta_{H}(v)= v \otimes 1 + (1-v) \otimes v, \quad \epsilon(v)=0.
\end{align*}
For each $v\in {^{g}(H_{(n)})^{1}}$ with $g\in \G(H)$ and $n\ge 1$, there exist elements $z_{v}\in H_{(0)}$, $v' \in H_{(n-1)}$ and $v'' \in {(H_{(n-1)})^{1}}$ such that $g=1-z_{v}\in \G(H)$, $\epsilon(v')=\epsilon(v'')=0$ and
\begin{align*}
\Delta_{H}(v) &= v \otimes 1 + (1-z_{v}) \otimes v + \sum_{} v' \otimes v''.
\end{align*}
For convenience, we set $H_{(-1)}=0$.  When $v= 1-u\in H_{(0)}$ with $u\in \G(H)$, we also adopt the notation $z_{v}=v$.

Next define
\begin{align*}
G &:= \{1-g \mid g\in \G(H)\}, \\
G_{\infty} &:= \{1-g \mid g\in \G(H) \text{ and } \ord(g)= + \infty\}.
\end{align*}
Clearly, $G_{\infty} \subseteq G \subseteq \ker \epsilon_{H} \cap H_{(0)}$. By Lemma \ref{lem:corad-pointed-Hopfalg}, we may split  $G_{\infty}$ into  two disjoint subsets $G_{\infty}^{+}, G_{\infty}^{-}$ so that $G_{\infty}^{-} = S_{H}(G_{\infty}^{+})$ (e.g., $1-g \in G_{\infty}^{+}$, $1-g^{-1} \in G_{\infty}^{-}$).

Note that the subcoalgebra generated by finitely many elements in $H$ is finite dimensional by \cite[Lemma 2.2.1]{Ra2012}. 
Hence,  we may choose a generating set $X$ of $H$ with the assignment $f_{X}: X \rightarrow H$, a subset $D \subseteq X$, a mirror subset $D^{*} $ and a map $t: X \rightarrow \mathbb{N}$ defined by 
$$t(x) = \min\{n \mid f_{X}(x) \in H_{(n)} \}$$ 
(see Example \ref{example:filtration-t-degree}), such that  
\begin{itemize}
\item $f_{X} (X) \subseteq \ker \epsilon_{H}$ and $f_{X}(X_{0}) = G$; 
\item $D, D^{*} \subseteq X_{0}$, $f_{X}(D) = G_{\infty}^{+}$ and $f_{X}(D^{*}) = G_{\infty}^{-}$; 
\item $\pi: \K\X \rightarrow H$ is the canonical projection induced by $f_{X}$;
\item the coproduct $\Delta_{H}$ lifts to an algebra map $\Delta:\K\X \rightarrow \K\X \otimes \K\X$ satisfying the following property:\\
for every $x\in X$, there exist elements $z_{x} \in X_{0}$, $x'\in \K\X^{+}$ and $x'' \in X$ such that
\begin{align*}
\Delta(x) = x \otimes 1 + (1-z_{x}) \otimes x + \sum_{t(x'')< t(x)}  x_{}' \otimes x'',
\end{align*}
where $\pi(1-z_{x})\in \G(H)$, and $z_{x_{0}} = x_{0} $ for all $x_{0}\in X_{0}$. 
\end{itemize}
Let $I:=\ker \pi$. By construction, $\Delta$ is a comultiplication on $\K\X$ such that 
$$\Delta(I) \subseteq I \otimes \K\X + \K\X \otimes I.$$
 Similarly,  the counit $\epsilon_{H}$ lifts to an algebra map $\epsilon: \K\X \rightarrow \K$ such that $X, I \subseteq \ker \epsilon$.

Let $O:= X\setminus D^{*}$. Then $X = O \cup D^{*}$.
We now define an order $\prec$ on $X$. For each $n\ge 0$, fix a well order $\prec_{n}$ on the set $X_{n} \cap O$. For $x,y\in O$, define
\begin{align*}
x <_{O} y \Longleftrightarrow  
\left\{\begin{array}{l}
t(x)< t(y), \text { or } \\
t(x)= t(y) \text{ and } 
x \prec_{t(x)} y.
\end{array}\right.
\end{align*}
Since $D,D^{*} \subseteq X_{0}$,  we have  $t(x^{*})=t(x)$ for all $x\in D$. \\
The order $\prec$ on $X$ induced by $<_{O}$ is then given as follows:  for $x, y \in X$, 
\begin{align*}
x \prec y \Longleftrightarrow  
\left\{\begin{array}{l}
t(x)< t(y), \text { or } \\
t(x)= t(y) \text{ and } 
\left\{\begin{array}{l}
o(x) \prec_{t(x)} o(y), \text{ or }\\
o(x) = o(y) =x \text{ and } y=x^{*}.
\end{array}\right.
\end{array}\right.
\end{align*}

It is easy to see that $(O,<_{O})$ and $(X,\prec)$ are well-ordered. 
Hence, the induced reduction order $\prec_{\R}$ is a well order on $\X$ by Lemma \ref{lem:reduction-order-well}, and the comultiplication  $\Delta$ is skew-triangular. Moreover, $I$ satisfies the additional assumption of Theorem \ref{thm:rl-sum-smaller-irrls} by Formulas (\ref{formula:invertible-relation-finite-1}) and (\ref{formula:invertible-relation-infinite-1}), since $\pi(1-x)\in \G (H)$  for every $x\in X_{0}$. 
Therefore, Part(a) follows from Theorem \ref{thm:rl-sum-smaller-irrls}. 

Now assume that $H$ is right Noetherian. By Lemma \ref{lem:Noetherian-irrls-finite}, the set $X_{I} \cap O$  is finite, and  by Lemma \ref{lem:relation-x-x*} (a) the set  $X_{I} \cap D^{*}$ is also finite. It follows that $X_{I}$ is finite. Therefore, $H$ is affine by Part(a). If $H$ is left Noetherian, then the Hopf algebra $H^{op~cop}$ is right Noetherian and hence affine. Therefore, $H$ is also affine.
\end{proof}

\begin{remark}\label{rmk:main-theorem-2}
(1) As a consequence,  any right or left Noetherian connected Hopf algebra is affine.
\noindent
(2) With the notation  of Theorem \ref{thm:noetherian-pointed-affine},  Theorem \ref{thm:rl-sum-smaller-irrls} implies that $\pi(X_{0, I})$ generates the coradical $H_{(0)}$ as an algebra, where $X_{0,I} := X_{0} \cap X_{I}$. Hence, if $H$ is  right Noetherian,  then $X_{0,I}$ is finite and therefore $H_{(0)}$ is affine. 
\end{remark}

We now consider the Noetherian and  affine properties of Hopf subalgebras of a Noetherian pointed Hopf algebra. 

Let $H$ be a Hopf algebra over $\K$, and $K$ a Hopf subalgebra of $H$. Suppose that $H_{(0)} \subseteq K$ or $H$ is pointed.
By \cite[Theorem 9.3.1]{Ra2012}, the Hopf algebra $H$  is a free left  $K$-module. Hence, $H$ is faithfully flat over $K$.  If $H$ is right Noetherian, then by \cite[Exercise 17T]{GW2004}, any left faithfully flat Hopf subalgebra of a right Noetherian Hopf algebra is itself right Noetherian. Therefore, $K$ is right Noetherian.  Thus, we obtain the following lemma. 

\begin{lemma}\label{lem:Hopfsubalg-Notherian}
Let $H$ be a Hopf algebra over $\K$ and $K$ a Hopf subalgebra of $H$. Suppose that $H_{(0)} \subseteq K$ or $H$ is pointed. If $H$ is right (resp. left) Noetherian, then $K$ is right (resp. left) Noetherian.
\end{lemma}

Lemma \ref{lem:Hopfsubalg-Notherian} provides a partial answer to Andruskiewitsch’s question \cite[Question 8]{A2023} on whether every Hopf subalgebra of a Noetherian Hopf algebra is itself Noetherian.

As a  consequence, every Hopf subalgebra of a Noetherian pointed Hopf algebra is affine.
\begin{corollary}\label{cor:Hopfsubalg-affine}
Let $H$ be a pointed Hopf algebra over $\K$. If $H$ is right or left Noetherian, then every Hopf subalgebra of $H$ is affine, or equivalently,  every ascending chain of Hopf subalgebras of $H$ stabilizes.
\end{corollary}
\begin{proof}
The proof follows from Lemma \ref{lem:Hopfsubalg-Notherian} and Theorem \ref{thm:noetherian-pointed-affine}.
\end{proof}

\begin{remark}\label{rmk:Hopfsubalg-affine}
This corollary generalizes a classical phenomenon from group and Lie theory:
if the group algebra of a group (respectively, the universal enveloping algebra of a Lie algebra) is Noetherian, then the group (respectively, the Lie algebra) satisfies the ACC on subgroups (respectively, subalgebras), or equivalently, all of its subgroups (respectively, subalgebras) are finitely generated.

Similarly, Corollary~\ref{cor:Hopfsubalg-affine} shows that if a pointed Hopf algebra 
$H$ is Noetherian, then every Hopf subalgebra of $H$ is affine. In particular, a pointed Hopf algebra cannot be Noetherian if it contains an infinitely generated Hopf subalgebra.
\end{remark}

\bibliographystyle{plain}

\end{document}